\documentclass[onefignum,onetabnum]{siamonline220329}

\usepackage{amsfonts}
\usepackage{amssymb}
\usepackage{graphicx}
\usepackage{bm}
\usepackage{epstopdf}
\usepackage{algorithm}
\usepackage{algorithmic}
\ifpdf
  \DeclareGraphicsExtensions{.eps,.pdf,.png,.jpg}
\else
  \DeclareGraphicsExtensions{.eps}
\fi
\usepackage{enumitem}
\setlist[enumerate]{leftmargin=.5in}
\setlist[itemize]{leftmargin=.5in}

\newcommand{\setR}[0]{\mathbb{R}}
\newcommand{\PP}[0]{\mathbb{P}}
\newcommand{\Pro}[0]{\mathcal{P}_{\Uad}}

\newcommand{\E}[0]{\mathbb{E}}
\newcommand{\w}[0]{\omega}
\newcommand{\B}[0]{\mathcal{B}}
\newcommand{\C}[0]{\mathcal{C}}
\newcommand{\LQ}[0]{\left[}
\newcommand{\RQ}[0]{\right]}
\newcommand{\D}[0]{\mathcal{D}}
\newcommand{\Uad}[0]{U_\text{ad}}
\newcommand{\alphab}[0]{{\bm \alpha}}
\newcommand{\QQ}[0]{\mathcal{Q}}
\newcommand{\us}[0]{u^\star}
\newcommand{\usMLMC}[1]{u^{\star,\overrightarrow{\w}_{#1}}}
\newcommand{\umq}[1]{u^{\text{MQ}(#1)}}

\newcommand{\betab}[0]{{\bm \beta}}

\newcommand{\xib}[0]{{\bm \xi}}

\newcommand{\N}[0]{\mathbb{N}}
\newcommand{\umlmc}[1]{u^{\text{MLMC}(#1)}}
\newcommand{\SL}[0]{\text{SL}}
\newcommand{\qbar}[0]{\overline{q}}

\newsiamremark{remark}{Remark}
\newsiamremark{example}{Example}
\newsiamremark{assumption}{Assumption}
\newsiamremark{hypothesis}{Hypothesis}
\crefname{hypothesis}{Hypothesis}{Hypotheses}
\newsiamthm{claim}{Claim}

\headers{ML quadrature formulae for OCP of random PDEs}{F. Nobile, T. Vanzan}

\title{Multilevel quadrature formulae for the optimal control of random PDEs}

\author{Fabio Nobile \thanks{CSQI Chair, Ecole Polytechnique F\'ed\'erale de Lausanne, Switzerland, 
  (\email{fabio.nobile@epfl.ch}).}
\and Tommaso Vanzan\thanks{Dipartimento di Scienze Matematiche, Politecnico di Torino, Italy, 
  (\email{tommaso.vanzan@polito.it}).}}

\ifpdf
\hypersetup{
  pdftitle={MLMC for OCP of random PDEs},
  pdfauthor={F. Nobile, T. Vanzan}
}
\fi
\begin{document}

\maketitle

\begin{abstract}
This manuscript presents a framework for using multilevel quadrature formulae to compute the solution of optimal control problems constrained by random partial differential equations.
Our approach consists in solving a sequence of optimal control problems discretized with different levels of accuracy of the physical and probability discretizations. The final approximation of the control is then obtained in a postprocessing step, by suitably combining the adjoint variables computed on the different levels. 
We present a general convergence and complexity analysis for an unconstrained linear quadratic problem under abstract assumptions on the spatial discretization and on the quadrature formulae. We detail our framework for the specific case of a MultiLevel Monte Carlo (MLMC) quadrature formula, and numerical experiments confirm the better computational complexity of our MLMC approach compared to a standard Monte Carlo sample average approximation, even beyond the theoretical assumptions.
\end{abstract}

\begin{keywords}
multilevel quadrature formulae, multilevel Monte Carlo, PDE-constrained optimization under uncertainty, random PDEs
\end{keywords}

\begin{MSCcodes}
35Q93, 49M41, 65C05, 35R60
\end{MSCcodes}

\section{Introduction}

In this manuscript, we are concerned with the numerical solution of Optimal Control Problems (OCPs) constrained by random Partial Differential Equations (PDEs). Concisely, the mathematical problem is 
\[\min_{u\in \Uad} J(u):=\E\LQ Q(y(u),u)\RQ,\]
where $u$ is the unknown \textit{deterministic} control, $y$ is the solution of a random PDE, $Q$ is quantity of interest, and $\mathbb{E}$ is the expectation operator.
OCPs under uncertainty have been extensively studied in the last decade since they are effective mathematical models to control physical processes that are affected by an intrinsic variability, or when only a partial knowledge of the system is available.

Theoretical foundations for the well-posedness of OCPs under uncertainty have been established in, e.g, \cite{kouri2018optimization,martinez2018optimal,Risk_adverse}. These analyses show that the evaluation of the gradient of $J$ requires the computation of the expectation of a properly defined adjoint variable over the probability distribution of the random inputs. In practice, a common solution strategy consists in a sample-based discretization, e.g., with Monte Carlo samples, of the underlying probability distribution \cite{shapiro2021lectures,milz2023sample}. Consequently, the expectation of the adjoint variable is replaced with a quadrature formula of, say, $N$ samples, and standard optimization algorithms can be used. However, a gradient descent iteration would generally require to solve $2N$ PDEs ($N$ state and $N$ adjoint equations) at each iteration: a cost which is rarely affordable. The need of effective strategies to reduce the computational cost of the gradient evaluation has motivated an active area of research.

Several contributions aim at improving the sampling strategy (i.e. reducing the number of samples while preserving the accuracy) by exploiting any regularity of the adjoint variable with respect to the random parameters. Examples are Quasi-Monte Carlo methods \cite{doi:10.1137/19M1294952}, and Stochastic Collocation methods based on tensorized Gaussian grids \cite[Appendix]{martin2021complexity}.
Another set of works approximate the gradient using multilevel quadrature formulae, such as Multilevel Monte Carlo (MLMC) \cite{van2019robust}, and Multilevel Quasi-Monte Carlo \cite{guth2023multilevel}. These techniques do not necessarily reduce the number of samples, but they exploit coarser discretizations to achieve an overall reduction of the computational cost.
More recently, \cite{beiser2023adaptive,ganesh2023gradient} proposed adaptive strategies that improve the accuracy of the quadrature formulae along the iterations, so that cheap gradient evaluations are used at the beginning, and accurate/expensive ones are needed only close to the optimum. These methods share similarities with batch-versions of stochastic gradient methods. The latters have been analyzed in the context of PDE-constrained optimization under uncertainty in, e.g.,  \cite{geiersbach2020stochastic,geiersbach2023stochastic}.

Notice that the multilevel/sparse approaches mentioned above are somehow limited to first-order, gradient based, optimization methods. As a matter of fact, a direct multilevel/sparse approximation of the expectation operator appearing in the objective function $J$ could, on the one hand, pave the way to higher order optimization algorithms, but on the other hand, may involve negative weights and therefore may destroy the (possible) convexity of the original OCP. A sparse grid approximation of the objective functional has been proposed in \cite{kouri2013trust,kouri2014inexact}, and the possible loss of convexity is handled through a trust region algorithm, which however is needed even for a simple linear-quadratic problem.

An alternative approach to use sparse and multilevel quadrature formulae has been recently proposed in \cite{nobile2024combination}, inspired by the so-called combination technique (CT) \cite{griebel1990combination,combination_te,pflaum1999error}.
The main idea of \cite{nobile2024combination} is to consider a hierarchical representation of the optimal control $u$ via,
\begin{equation}\label{eq:CT}
u\approx u_{CT}=\sum_{(\alphab,\betab)\in \mathcal{I}} c_{\alphab,\betab} u^{\alphab,\betab},
\end{equation}
where $\alphab,\betab$ are multi-indices related to the level of accuracy of a finite element discretization in space and of a Stochastic Collocation quadrature in probability, $u^{\alphab,\betab}$ is the minimizer of the OCP discretized with accuracy level $(\alphab,\betab)$, and $\mathcal{I}$ is a suitable set of sparse multi-indices. In other words, the approximation defined by \eqref{eq:CT} consists in solving a \textit{sequence} of (relatively cheap) minimization problems on tensor product type discretizations, and to combine the computed controls in a post-processing step to obtain the final approximation $u_{CT}$. Assuming that $u_{CT}$ converges to $u$ when $\mathcal{I}$ exhausts the multi-indices, and under some strong regularity assumptions, \cite{nobile2024combination} shows that the CT approximation 
drastically reduces the computational cost with respect to a tensor product approximation $u^{\alphab,\betab}$ of the same accuracy.
The present manuscript greatly generalizes the method and theory proposed in \cite{nobile2024combination} in several directions:
\begin{itemize}
\item We present a general framework that accommodates any (possibly randomized) multilevel quadrature formulae (hence, not restricted to the stochastic collocation method). Consequently, the new framework is more flexibile, for instance by resorting to a multilevel Monte Carlo quadrature, to tackle optimization problems that 1) do not admit an easy parametrization of the randomness through a set of independent random variables 2) have a low-regularity dependence on the the random inputs which, moreover, may be high-dimensional. 

\item Under very mild assumptions on the finite element discretization and on the quadrature formula, we present a complete convergence analysis of our multilevel framework in the case of an unconstrained linear-quadratic problem, with possibly boundary/local controls and observations. This is an important theoretical contribution since, a-priori, it is not evident that the combination of minimizers computed with different levels of accuracy converges to the exact minimizer. We further analyse the complexity of our multilevel framework and compare it with that of a standard, single level, approach. 

\item The novel approach can handle control constraints that can be modeled through closed and convex subsets of the control space, despite our convergence analysis does not cover at the moment this case. The main idea is to suitably combine adjoint variables instead of the controls, as originally proposed in \cite{nobile2024combination}.

\item We detail our new framework for the specific choice of a Monte Carlo quadrature formulae, that leads to a MLMC strategy. We present further numerical experiments that show the efficacy of the MLMC even beyond the hypotheses of the theoretical analysis, in particular with control constraints.
\end{itemize}

We wish to emphasize that our framework does not represent a new optimization algorithm to solve OCPs under uncertainty, rather it provides a way to smartly combine minimizers computed for different levels of spatial and probability accuracy. Notably, it is agnostic to the specific choice of the optimization algorithm used to solve the sequence of minimization problems, so that state-of-the-art optimization algorithms (e.g., \cite{nocedal1999numerical,hinze2008optimization,Kouri2018,kouri2022primal,antil2023alesqp}) and inner linear solvers (e.g., \cite{Kouri2018,vanzan,ciaramella2024multigrid}) can be used.

Further, the convergence analysis is based on a nonstandard multilevel study, since on each spatial level $\ell$, the adjoint variable $p$ is computed from a bilinear form (the optimality system) which depends itself on the law of $p$. Thus, an ad-hoc treatment is needed, which may be extended in future endeavours to study multilevel methods applied to random PDEs containing nonlocal terms with respect to the random variables.

The manuscript is organized as follows: Section \ref{sec:problem_formulation} defines the model problem and its numerical discretization.
Section \ref{sec:multilevelquadrature_formulae} introduces the new framework considering a general multilevel quadrature formula. Section \ref{sec:convergence_analysis} presents its convergence analysis under suitable assumptions. Section \ref{sec:Complexity} discusses the computational cost of our multilevel framework and details one particular instance that leads to a multilevel Monte Carlo method. Finally, numerical experiments are discussed in Section \ref{sec:numerical_section}, which show the better complexity of the proposed framework compared to a standard Monte Carlo sample average approximation.

\section{Problem formulation}\label{sec:problem_formulation}
Let $\D\subset \mathbb{R}^d$ be a Lipschitz bounded domain, $V$ a Sobolev space (e.g., $H^1(\D)$ equipped with suitable boundary conditions), and $(\Omega,\mathcal{F},\PP)$ a complete probability space. We consider the linear elliptic random PDE,
\begin{equation}\label{eq:random_PDE_weak}
a_\omega(y_\w,v)=\langle f,v\rangle_{V',V},\quad\forall v\in V,\quad \PP\text{-a-e. } \omega\in \Omega,
\end{equation}
where $a_{\omega}(\cdot,\cdot):V\times V\rightarrow \mathbb{R}$ is a continuous and coercive bilinear form for $\PP\text{-a.e. }\omega$ (with continuity and coercivity constants possibly dependent on $\w$), $f$ is an element in the dual of $V$, denoted $V^\prime$, and $\langle\cdot,\cdot \rangle_{V',V}$ denotes the duality between $V^\prime$ and $V$.
Associated to \eqref{eq:random_PDE_weak} we define, for $\PP\text{-a.e. }\omega\in \Omega$, the linear solution operator $S^\w: f\in V^\prime \rightarrow S^\w f\in V$, where 
\[a_\omega(S^\w f,v)=\langle f,v\rangle_{V',V},\quad \forall v\in V.\]
$S^\w$ is a continuous operator and we suppose its continuity constant $C_S(\w)$ lies in $L^q(\Omega;\setR)$ for any $q\in [1,\infty)$ with respect to the probability measure $\mathbb{P}$ (see, e.g., \cite{lord_powell_shardlow_2014,Scheichl} for sufficient conditions on $a_\w(\cdot,\cdot)$ for this to hold).

\noindent In this manuscript, we consider the minimization of functionals constrained by \eqref{eq:random_PDE_weak}.
Our model problem is
\begin{equation}\label{eq:OCP_model_problem}
\begin{aligned}
&\min_{u\in \Uad ,y\in L^2(\Omega;V)} \widehat{J}(y,u):= \frac{1}{2}\E\LQ \|\C y-y_d\|^2_H\RQ +\frac{\nu}{2}\|u\|^2_{U},\\
&\quad\text{subject to}\\
&a_\omega(y_\omega,v)=\langle \B u,v\rangle_{V',V},\quad \forall v \in V,\ \PP\text{-a.e. } \omega\in \Omega,
\end{aligned}
\end{equation}
where $\Uad$ is a closed and convex subset of a Hilbert space $U$ and 
$\B:U\rightarrow V^\prime$
is a linear and continuous control operator allowing possibly for a local control (i.e. a control acting only on a subset $\D_0\subset \D)$ or a boundary control (i.e. a control acting as Neumann condition on a subset of $\partial \D$). The Hilbert space $H$ is the space of observations, $y_d\in H$ is the target state, $\C:V\rightarrow H$ is a linear and continuous observation operator, $\E$ is the expectation operator, and $\nu$ is a positive scalar.
Introducing the linear control-to-state map $\widehat{S}:V^\prime \rightarrow L^2(\Omega;V)$, with $(\widehat{S}g)(\w)=S^\omega g$, $\forall g\in V^\prime$ and for $\PP\text{-a.e. }\omega$, 
the reduced formulation of \eqref{eq:OCP_model_problem} is
\begin{equation}\label{eq:OCP}
\min_{u\in \Uad} J(u):=\widehat{J}(\widehat{S}\B u,u).
\end{equation}
Existence and uniqueness of the minimizer of \eqref{eq:OCP} follows directly from standard variational arguments \cite{lions1971optimal,hinze2008optimization,troltzsch2010optimal,antil2018brief,kouri2018optimization}. 
Further, due to the linearity of the PDE constraint, $J$ is Fr\'echet differentiable, and its Fr\'echet derivative at a point $u$ along a direction $v\in U$ is the linear functional $J^\prime(u):U \rightarrow \setR$ equal to
\[\langle J^\prime(u),v\rangle_{U',U} = \langle \nu \Lambda_U u-\B^\star \E\LQ S^{\w,\star}(\C^\star \Lambda_H(y_d-\C S^{\w}(\B u)))\RQ,v\rangle_{U',U},\quad \forall v\in U\]
where $\Lambda_U$ and $\Lambda_H$ are the Riesz isomorphisms from $U$ to $U^\prime$ and from $H$ to $H^\prime$, $\B^\star$ and $\C^\star$ are the adjoints of $\mathcal{B}$ and $\C$ respectively, and $S^{\w,\star}$ is the adjoint operator of $S^\w$ satisfying $\forall g\in V^\prime$ and for $\PP\text{-a.e. } \omega$,
\[ a_\omega(v,S^{\w,\star}g)=\langle g,v\rangle_{V',V},\quad \forall v \in V.\]
The minimizer $\us$ of \eqref{eq:OCP} satisfies the optimality condition (\cite[Theorem 1.41]{hinze2008optimization}),
\begin{equation}\label{eq:inequality}
\langle J^\prime(\us), w-\us\rangle_{U',U} \geq 0,\quad \forall w\in \Uad,
\end{equation}
and since $\Uad$ is convex and closed, the variational inequality \eqref{eq:inequality} may be reformulated as the nonlinear equation (see, e.g., \cite[Corollary 1.2]{hinze2008optimization})
\begin{equation}\label{eq:optimality_condition}
\us =\Pro\left(\frac{1}{\nu} \Lambda_U^{-1} \B^\star\E\LQ p(\us)\RQ\right),
\end{equation}
where to ease the notation we have introduced the adjoint variable $p(\us)\in L^2(\Omega;V)$, defined as $p^\w(\us):=S^{\w,\star}(\C^\star\Lambda_H(y_d-\C S^{\w}(\B \us)))$, and $\Pro$ represents the projection onto the set of admissible controls $\Uad$. Both from the theoretical and numerical point of view, it is useful to remark that \eqref{eq:optimality_condition} admits the equivalent full-space formulation, obtained by explicitly stating the dependence of $p(\us)$ on $\us$,
\begin{equation}\label{eq:optimality_system}
\begin{aligned}
a_\omega(y^\w,v)&=\langle \B \us,v\rangle_{V',V},\quad\forall v\in V,\quad \PP\text{-a-e. } \omega\in \Omega,\\
a_\omega(v,p^\w)&=\langle \C^\star\Lambda_H(y_d-\C y^\w) ,v\rangle_{V',V},\quad\forall v\in V,\quad \PP\text{-a-e. } \omega\in \Omega,\\
\us&=\Pro\left(\frac{1}{\nu}\Lambda_U^{-1}\B^\star\E\LQ p\RQ\right).
\end{aligned}
\end{equation}

\subsection{Numerical Discretization}\label{sec:numerical_discretization}

The numerical solution of \eqref{eq:OCP_model_problem} requires two distinct numerical approximations: a finite dimensional approximation of the solution operator $S$, and a suitable quadrature formula for the expectation operator $\E$. In the next paragraphs, we detail these two components and set up the notation that will ease the understanding of the novel multilevel approximation proposed.

Let $\left\{\mathcal{T}_\ell\right\}_{\ell\geq 0}$ be a family of regular triangulations of $\mathcal{D}$ of mesh sizes $\left\{h_{\ell}\right\}_{\ell\geq 0}$,  with $\tau:=\min_{\ell} \frac{h_{\ell-1}}{h_{\ell}}<\infty$, and $\left\{V_{\ell}\right\}_{\ell\geq 0}$ be an associated sequence of finite dimensional subspaces of $V$.
For any $V_\ell$, we define finite dimensional approximations of $S^\w$ and $S^{\w,\star}$, that is, $S^\w_{\ell}: V^\prime \rightarrow V_{\ell}$ and $S^{\w,\star}_{\ell}: V^\prime \rightarrow V_{\ell}$, satisfying
\begin{equation*}
\begin{aligned}
a_{\w}(S^\w_\ell g,v_\ell)=\langle g,v_\ell\rangle_{V',V},\quad \forall v_\ell \in V_\ell,\quad \mathbb{P}\text{-a.e.}\;\w\in \Omega,\\
a_{\w}(v_\ell,S^{\w,\star}_\ell g)=\langle g,v_\ell\rangle_{V',V},\quad \forall v_\ell \in V_\ell,\quad \mathbb{P}\text{-a.e.}\;\w\in \Omega.
\end{aligned}
\end{equation*}
Note that we do not consider a sequence of finite element subspaces of the control space $U$, but we follow the variational discretization principle (\cite[Chapter 3]{hinze2008optimization}), so that the control $u$ is implicitly discretized by the choice of $S^\w_\ell$.

Next, let $\left\{\QQ^k \right\}_{k\geq 0}$ be a sequence of quadrature formulae of increasing accuracy. The quadrature nodes of $\QQ^k$ are denoted by $\left\{\xi^k_j\right\}_{j=1}^{N_k}$ and the quadrature weights by $\left\{\zeta^k_j\right\}_{j=1}^{N_k}$. We assume that all weights are positive. For every random variable $X\in L^1(\Omega;K)$, $K$ being a generic Banach space, the expected value of X can be approximated by
\[\E\LQ X\RQ \approx \QQ^k\LQ X\RQ =\sum_{j=1}^{N_k} \zeta^k_j X(\xi^k_j).\] 

\noindent This general setting includes Monte Carlo methods, whose quadrature nodes are drawn randomly and the weights are all equal to $\zeta_j^k=\frac{1}{N_k}$ for every $k\geq 0$.
Further, if we assume that the probability space can be parametrized by a set of $M$ random variables, we may also consider (possibly randomized) Quasi-Monte Carlo methods as well as collocation methods based on Gaussian quadrature formulae. We however restrict our analysis to quadrature formulae with positive weights, thus excluding sparse grids, since negative quadrature weights may destroy the convexity of \eqref{eq:OCP_model_problem}.

To compute a numerical approximation of $\us$, we can solve for given and fixed values of $\ell$ and $k$ the optimization problem,
\begin{equation}\label{eq:OCP_kell}
\min_{u\in U_{ad}} J_{\ell,k}(u):=\frac{1}{2}\QQ^k \LQ \|\C S^\w_{\ell}\B u-y_d\|_H^2 \RQ +\frac{\nu}{2}\|u\|^2_{U}.
\end{equation}
We denote the unique minimizer of \eqref{eq:OCP_kell} by $\us_{\ell,k}$, and remark that it satisfies the optimality conditions
\begin{align}
\langle J_{\ell,k}^\prime(\us_{\ell,k}), w-\us_{\ell,k}\rangle_{U',U} \geq 0&,\quad \forall w\in \Uad,\label{eq:inequality_disc}\\
\us_{\ell,k} =\Pro\left(\frac{1}{\nu} \Lambda_U^{-1} \B^\star\QQ^k\LQ p_\ell(\us_{\ell,k})\RQ\right)&\label{eq:opt_condition},
\end{align}
where $p^\w_{\ell}(u):=S_\ell^{\w,\star}(\C^\star \Lambda_H(y_d-\C S^\w_\ell(\B u)))$ is the adjoint variable defined through the finite dimensional operators $S_\ell^\w$ and $S^{\w,\star}_\ell$. Similarly to \eqref{eq:optimality_system}, \eqref{eq:opt_condition} can be formulated as the fully-discrete, full-space nonlinear system,
\begin{equation}\label{eq:full_space_system_discretized}
\begin{aligned}
a_{\xi^k_j}(y^{\xi^k_j}_\ell,v)&=\langle \B u_{\ell,k},v\rangle_{V',V},\quad\forall v\in V_\ell,\quad \forall \xi_j^k,
\; j=1,\dots,N_k,\\
a_{\xi^k_j}(v,p^{\xi^k_j}_\ell)&=\langle \C^\star\Lambda_H(y_d-\C y^{\xi^k_j}_\ell) ,v\rangle_{V',V},\quad\forall v\in V_\ell,\quad \forall \xi_j^k,\; j=1,\dots,N_k,\\
u_{\ell,k}&=\Pro\left(\frac{1}{\nu}\Lambda_U^{-1}\B^\star\QQ^k\LQ p_\ell\RQ\right).
\end{aligned}
\end{equation}
From the practical point of view, \eqref{eq:opt_condition} paves the way to solution strategies based on (projected) gradient descent methods, which at the  $i$-th iteration require to solve $N_k$ state and $N_k$ adjoint finite element problems set on the space $V_\ell$ to compute $\QQ^k\LQ p(u^i_{\ell,k})\RQ$, or semismooth Newton's algorithms. Alternatively, \eqref{eq:full_space_system_discretized} is suitable for full-space solvers, which require only to precondition the $2N_k$ linear systems, but may suffer from high memory storage. In both cases, there is a great need to reduce the overall computational cost. 
In the next section, we propose our solution which will leverage different levels of discretizations of both physical and probability spaces to achieve a better computational complexity than the classical discretization and solution paradigm that we have here recalled.

We close this section by introducing two auxiliary optimization problems which will naturally arise in our analysis. They are obtained by performing either a spatial or a probability semidiscretization of \eqref{eq:OCP}, namely
\begin{equation}\label{eq:OCP_ell}
\min_{u\in \Uad} J_{\ell,\infty}(u):=\frac{1}{2}\E \LQ \|\C S^\w_\ell \B u -y_d\|_H^2 \RQ +\frac{\nu}{2}\|u\|^2_{U},
\end{equation}
and
\begin{equation}\label{eq:OCP_k}
\min_{u\in \Uad} J_{\infty,k}(u):=\frac{1}{2}\QQ^k \LQ \|\C S^\w \B u -y_d\|_H^2 \RQ +\frac{\nu}{2}\|u\|^2_{U}.
\end{equation}
We denote the minimizers of \eqref{eq:OCP_ell} and \eqref{eq:OCP_k} by $\us_{\ell,\infty}$ and $\us_{\infty,k}$, respectively.

\section{Multilevel quadrature formulae}\label{sec:multilevelquadrature_formulae}
To define our multilevel approximation, we start by introducing the univariate details
\begin{align*}
\Delta^s \QQ^k\LQ p_\ell(\us_{\ell,k})\RQ&:=\QQ^k\LQ p_\ell(\us_{\ell,k})\RQ-\QQ^k\LQ p_{\ell-1}(\us_{\ell-1,k})\RQ,\\
\Delta^q \QQ^k\LQ p_\ell(\us_{\ell,k})\RQ&:=\QQ^k\LQ p_\ell(\us_{\ell,k})\RQ-\QQ^{k-1}\LQ p_{\ell}(\us_{\ell,k-1})\RQ,
\end{align*}
which compute the difference between the approximated quadrature of the adjoint variable, evaluated at the optimal control, obtained with two consecutive levels of spatial and quadrature discretizations, respectively.
Then, the so-called hierarchical surpluses $\Delta \QQ^k\LQ p_\ell(\us_{\ell,k})\RQ $ are defined for every $\ell\geq 0$ and $k\geq 0$ as 
\begin{align*}\label{eq:definition_delta}
\Delta \QQ^k& \LQ p_\ell(\us_{\ell,k})\RQ := \Delta^s\Delta^q \QQ^k\LQ p_\ell(\us_{\ell,k})\RQ\\
&=\QQ^k\LQ p_\ell(\us_{\ell,k})\RQ - \QQ^{k}\LQ p_{\ell-1}(\us_{\ell-1,k})\RQ - \QQ^{k-1}\LQ p_\ell(\us_{\ell,k-1})\RQ + \QQ^{k-1}\LQ p_{\ell-1}(\us_{\ell-1,k-1})\RQ,
\end{align*}
with the notation that $\QQ^{-1}\LQ p_{\ell}(\us_{\ell,-1})\RQ=\QQ^{k}\LQ p_{-1}(\us_{-1,k})\RQ=0$, $\forall k,\ell\geq 0$. Letting $L$ be an integer, a direct calculation shows that
\begin{align*}
\sum_{\ell=0}^L\sum_{k=0}^L \Delta \QQ^k\LQ p_\ell(\us_{\ell,k})\RQ&=\sum_{\ell=0}^L\sum_{k=0}^L \left(\QQ^k\LQ p_\ell(\us_{\ell,k})\RQ - \QQ^{k-1}\LQ p_\ell(\us_{\ell,k-1})\RQ\right)\\
& -\sum_{\ell=0}^L\sum_{k=0}^L \left( \QQ^{k}\LQ p_{\ell-1}(\us_{\ell-1,k})\RQ - \QQ^{k-1}\LQ p_{\ell-1}(\us_{\ell-1,k-1})\RQ\right)\\
&=\sum_{\ell=0}^L \left( \QQ^{L}\LQ p_{\ell}(\us_{\ell,L})\RQ - \QQ^{L}\LQ p_{\ell-1}(\us_{\ell-1,L})\RQ\right) =  \QQ^L\LQ p_L(\us_{L,L})\RQ.
\end{align*}
which allows us to reformulate the optimality condition \eqref{eq:opt_condition} as
\begin{equation}\label{eq:optimality_condition_revised}
\us_{L,L} =\Pro\left(\frac{1}{\nu}\Lambda_U^{-1}\B^\star\QQ^L\LQ p_L(\us_{L,L})\RQ\right)=\Pro\left(\frac{1}{\nu}\Lambda_U^{-1}\B^\star\sum_{\ell=0}^L\sum_{k=0}^L \Delta \QQ^k\LQ p_\ell(\us_{\ell,k})\RQ\right).
\end{equation}
Equation \eqref{eq:optimality_condition_revised} expresses the optimal control $\us_{L,L}$ as the projection on $\Uad$ of a combination of several adjoint variables $p_{\ell}(\us_{\ell,k})$, which, together with their quadrature $\QQ^k\LQ p_{\ell}(\us_{\ell,k})\RQ$, can be computed by solving the associated discretized optimal control \eqref{eq:OCP_kell} for each instance of $k$ and $\ell$. Although at first \eqref{eq:optimality_condition_revised} is not computationally attractive (to obtain $\us_{L,L}$ it is computationally cheaper to minimize $J_{L,L}$ once and for all, instead of minimizing $J_{k,\ell}$ for every $0\leq k,\ell\leq L$), \eqref{eq:optimality_condition_revised} represents the starting point of our novel approximation. Indeed,
it can be interpreted as a full tensor approximation over the product space of spatial approximations and quadrature formulae \cite{harbrecht2012multilevel}.
Hence, it is natural to consider a sparse approximation, obtained by constraining the indices of the space and probability discretizations to a different set. In this work, we restrict the summation over the set of multi-indices $\left\{(l,k)\in \N^2: 0\leq \ell+k\leq L\right\}$, obtaining the Multilevel Quadrature approximation of level $L$, denoted by $\umq{L}$, that is defined as
\begin{equation}\label{eq:multilevel_approx}
\medmuskip=-1mu
\thinmuskip=-1mu
\thickmuskip=-1mu
\nulldelimiterspace=0.9pt
\scriptspace=0.9pt 
\arraycolsep0.9em 
\begin{aligned}
\us_{L,L}\approx \umq{L}&:=\Pro\left(\frac{1}{\nu}\Lambda_U^{-1}\B^\star\sum_{\ell=0}^L\sum_{k=0}^{L-\ell} \Delta \QQ^k\LQ p_\ell(\us_{\ell,k})\RQ\right)\\
&=\Pro\left(\frac{1}{\nu}\Lambda_U^{-1}\B^\star\sum_{\ell=0}^L \QQ^{L-\ell}\LQ p_{\ell}(\us_{\ell,L-\ell})-p_{\ell-1}(\us_{\ell-1,L-\ell})\RQ \right)\\
&=\Pro\left(\frac{1}{\nu}\Lambda_U^{-1}\B^\star\left(\QQ^{L}\LQ p_{0}(\us_{0,L})\RQ + \sum_{\ell=1}^L \QQ^{L-\ell}\LQ p_{\ell}(\us_{\ell,L-\ell})-p_{\ell-1}(\us_{\ell-1,L-\ell})\RQ \right)\right).
\end{aligned}
\end{equation}

The expression $\umq{L}$ represents the multilevel approximation for the optimal control $\us$ that we propose and study in this manuscript. We will present a detailed convergence analysis assuming that $\left\{\QQ^k\right\}_{k\geq 0}$ is a sequence of unbiased, statistically independent, randomized quadrature formulae, since the theoretical analysis is more challenging in this setting. As the analysis develops, we discuss how the results can be framed in the context of deterministic quadrature formulae.
 
\noindent The next proposition provides a bound for the error associated to $\umq{L}$.
In view of the analysis developed in Section \ref{sec:convergence_analysis}, from now on we assume that the control operator $\B$ is a continuous linear functional from $U$ to $Z^\prime$, where $Z$ is a Hilbert space such that $V\subsetneq Z$, and we set $C_{\B}:=\|\B\|_{\mathcal{L}(U,Z^\prime)}$ \footnote{For several problems of interest, $\B$ defines a functional not only on $V$, but also on a larger space $Z$ (e.g., the classical distributed control $\langle \B u,v\rangle:=\int_\D u v$ is a well-defined functional over $L^2(\D)\supset H^1(\D)$). This additional hypothesis will be needed to assume convergence rates for the spatial finite element discretizations.}
\begin{proposition}\label{lemma:error_bound}
Let $\left\{\QQ^k\right\}_{k\geq 0}$ be a sequence of unbiased, statistically independent, randomized quadrature formulae. Then, the multilevel approximation $\umq{L}$ defined in \eqref{eq:multilevel_approx} satisfies the error bound,
\begin{equation}\label{eq:lemma_error_statement}
\medmuskip=-1mu
\thinmuskip=-1mu
\thickmuskip=-1mu
\nulldelimiterspace=0.9pt
\scriptspace=0.9pt 
\arraycolsep0.9em 
\begin{aligned}
\mathbb{E}\LQ \|\us -\umq{L}\|^2_U\RQ & \leq 2 \|\us-\us_{L,\infty}\|_U^2\\
& +
\frac{4C^2_{\B}}{\nu^2}\bigg(\sum_{\ell=0}^L\E\LQ \|\left(\mathbb{E}-\QQ^{L-\ell}\right)(p_{\ell}(\us_{\ell,\infty})-p_{\ell-1}(\us_{\ell-1,\infty})\|_Z^2\RQ \\
& + L\sum_{\ell=0}^L\E\LQ \QQ^{L-\ell}\|p_{\ell}(\us_{\ell,\infty})-p_{\ell}(\us_{\ell,L-\ell})-p_{\ell-1}(\us_{\ell-1,\infty})+p_{\ell-1}(\us_{\ell-1,L-\ell})\|_Z^2\RQ\bigg).
\end{aligned}
\end{equation}
\end{proposition}
\begin{proof}
We start by adding and substracting $\us_{L,\infty}$, that is, the minimizer of \eqref{eq:OCP_ell},
\begin{equation}
\E\LQ \|\us-\umq{L}\|^2_U\RQ \leq 2\|\us-\us_{L,\infty}\|^2_U+2\E\LQ \|\us_{L,\infty}-\umq{L}\|^2_U\RQ,
\end{equation}
and we then focus on the second term.
Using that the projector $\Pro$ is a nonexpansive map (\cite[Lemma 1.10]{hinze2008optimization}), the telescopic identity $\E\LQ p_{L}(\us_{L,\infty})\RQ=\sum_{\ell=0}^L \E\LQ p_{\ell}(\us_{\ell,\infty})-p_{\ell-1}(\us_{\ell-1,\infty})\RQ$, and the continuity of $\Lambda_U^{-1}$ and of $\B^\star$, we get
\begin{equation}\label{eq:first_eq_thm}
\medmuskip=-1mu
\thinmuskip=-1mu
\thickmuskip=-1mu
\nulldelimiterspace=0.9pt
\scriptspace=0.9pt 
\arraycolsep0.9em 
\begin{aligned}
&\E\LQ \|\us_{L,\infty}-\umq{L}\|^2_U\RQ\\
&=\E\LQ \left\|\Pro\left(\frac{\Lambda_U^{-1}\B^\star}{\nu}\E\LQ p_{L}(\us_{L,\infty})\RQ\right) -\Pro\left(\frac{\Lambda_U^{-1} \B^\star}{\nu}\sum_{\ell=0}^L \QQ^{L-\ell}\LQ p_{\ell}(\us_{\ell,L-\ell})-p_{\ell-1}(\us_{\ell-1,L-\ell})\RQ \right)\right\|^2_U\RQ\\
&\leq \E\LQ \left\|\frac{1}{\nu}\Lambda_U^{-1} \B^\star\left( \sum_{\ell=0}^L \E\LQ p_{\ell}(\us_{\ell,\infty})-p_{\ell-1}(\us_{\ell-1,\infty})\RQ -\QQ^{L-\ell}\LQ p_{\ell}(\us_{\ell,L-\ell})-p_{\ell-1}(\us_{\ell-1,L-\ell})\RQ\right)\right\|^2_U\RQ\\
&\leq \frac{2C^2_{\B}}{\nu^2}\Bigg( \E\LQ\left\|\sum_{\ell=0}^L \E\LQ p_{\ell}(\us_{\ell,\infty})-p_{\ell-1}(\us_{\ell-1,\infty})\RQ -\QQ^{L-\ell}\LQ p_{\ell}(\us_{\ell,\infty})-p_{\ell-1}(\us_{\ell-1,\infty})\RQ\right\|^2_Z\RQ\\
&+\E\LQ \left\|\sum_{\ell=0}^L \QQ^{L-\ell} \LQ p_{\ell}(\us_{\ell,\infty})-p_{\ell-1}(\us_{\ell-1,\infty})
-p_{\ell}(\us_{\ell,L-\ell})+p_{\ell-1}(\us_{\ell-1,L-\ell})\RQ\right\|^2_Z\RQ\Bigg).
\end{aligned}
\end{equation}
Next, using the statistical independence of the quadrature formulae across the levels and since $\E\LQ \QQ^{L-\ell} \LQ p_{\ell}(\us_{\ell,\infty})-p_{\ell-1}(\us_{\ell-1,\infty})\RQ \RQ = \E\LQ p_{\ell}(\us_{\ell,\infty})-p_{\ell-1}(\us_{\ell-1,\infty})\RQ$ (which holds true because $\QQ^{L-\ell}$ is unbiased and the adjoint variables are evaluated at the minimizers of continuous OCPs in probability), we have
\begin{equation}
\medmuskip=-1mu
\thinmuskip=-1mu
\thickmuskip=-1mu
\nulldelimiterspace=0.9pt
\scriptspace=0.9pt 
\arraycolsep0.9em 
\E\LQ\left\|\sum_{\ell=0}^L \left(\E-\QQ^{L-\ell}\right)\LQ p_{\ell}(\us_{\ell,\infty})-p_{\ell-1}(\us_{\ell-1,\infty})\RQ\right\|^2_Z\RQ=\sum_{\ell=0}^L \E\LQ\left\|\left(\E-\QQ^{L-\ell}\right)\LQ p_{\ell}(\us_{\ell,\infty})-p_{\ell-1}(\us_{\ell-1,\infty})\RQ\right\|_Z^2\RQ.
\end{equation}
On the other hand, since $\E\LQ p_{\ell}(\us_{\ell,\infty})\RQ \neq \E\LQ \QQ^{L-\ell}\LQ p_{\ell}(\us_{\ell,L-\ell})\RQ\RQ$, we can bound the last term using twice the Cauchy-Schwarz inequality,
\begin{equation}
\begin{aligned}
&\E\LQ \left\|\sum_{\ell=0}^L \QQ^{L-\ell} \LQ p_{\ell}(\us_{\ell,\infty})-p_{\ell-1}(\us_{\ell-1,\infty})
-p_{\ell}(\us_{\ell,L-\ell})+p_{\ell-1}(\us_{\ell-1,L-\ell})\RQ\right\|^2_Z\RQ\\
&\leq \E\LQ L \sum_{\ell=0}^L \left\| \QQ^{L-\ell} \LQ p_{\ell}(\us_{\ell,\infty})-p_{\ell-1}(\us_{\ell-1,\infty})
-p_{\ell}(\us_{\ell,L-\ell})+p_{\ell-1}(\us_{\ell-1,L-\ell})\RQ\right\|^2_Z\RQ\\
&=L \sum_{\ell=0}^L \E\LQ \left\| \sum_{j=1}^{N_{L-\ell}} \zeta_j^{L-\ell} \left(p^{\xi_j^{L-\ell}}_{\ell}(\us_{\ell,\infty})-p^{\xi_j^{L-\ell}}_{\ell-1}(\us_{\ell-1,\infty})
-p^{\xi_j^{L-\ell}}_{\ell}(\us_{\ell,L-\ell})+p^{\xi_j^{L-\ell}}_{\ell-1}(\us_{\ell-1,L-\ell})\right)\right\|^2_Z\RQ\\
&\leq L\sum_{\ell=0}^L \E\LQ \QQ^{L-\ell} \left\| p_{\ell}(\us_{\ell,\infty})-p_{\ell-1}(\us_{\ell-1,\infty})
-p_{\ell}(\us_{\ell,L-\ell})+p_{\ell-1}(\us_{\ell-1,L-\ell})\right\|^2_Z\RQ,
\end{aligned}
\end{equation}
where the last step uses that the weights $\left\{\zeta_j\right\}_{j=1}^{N_L-\ell}$ are positive and sum up to one.
\end{proof}
\begin{remark}[Proposition \ref{lemma:error_bound} for deterministic quadrature formulae]\label{remark:deterministic_quadrature}
If the quadrature formula is deterministic, the proof of Lemma \ref{lemma:error_bound} can be adapted by using repeatedly the triangle inequality. The final error estimate is
\begin{equation}\label{eq:err_bound_deterministic}
\begin{aligned}
\|\us-\umq{L}\|_U&\leq  \|\us-\us_{L,\infty}\|_U+ \frac{C_{\B}}{\nu}\sum_{\ell=0}^L \left\|(\E-\QQ^{L-\ell})\LQ p_{\ell}(\us_{\ell,\infty})-
p_{\ell-1}(\us_{\ell-1,\infty})\RQ\right\|_Z\\
&+\frac{C_{\B}}{\nu}\sum_{\ell=0}^L \QQ^{L-\ell} \LQ\|p_{\ell}(\us_{\ell,\infty})-p_{\ell}(\us_{\ell,L-\ell})-p_{\ell-1}(\us_{\ell-1,\infty})+p_{\ell-1}(\us_{\ell-1,L-\ell})\|_Z\RQ.
\end{aligned}
\end{equation}
\end{remark}

The error estimates \eqref{eq:lemma_error_statement} and \eqref{eq:err_bound_deterministic} are made of three terms. The first one is the bias term due to the spatial discretization of the OCP. The second one is the standard statistical term typically arising in multilevel analyses.
The third term instead is novel and it requires a detailed analysis which is presented in Section \ref{sec:convergence_analysis}. It appears since on each spatial level $\ell$, the approximated quadrature formula of level $L-\ell$ enters into the definition of the nonlinear optimality conditions that are solved (see \eqref{eq:opt_condition} or  \eqref{eq:full_space_system_discretized}). 
From this point of view, our analysis may be extended to study multilevel methods for  random PDEs containing terms that depend on the law of the random solution, e.g.,
\[-\nabla\cdot \left(\kappa(x,\w)\nabla y(x,\w)\right) +\E\LQ y(x,\w)\RQ = f(x,\w)\quad \forall x\in\D,\;\PP\text{-a.e. }\w.\]
In the context of stochastic differential equations, a prototype example is the McKean-Vlasov equation.

We now discuss some implementation aspects. Algorithm \ref{alg:MLQ} details how $\umq{L}$ is computed: starting from a coarse mesh of level $\ell=0$ and a fine quadrature formula of level $L$, we solve the corresponding minimization problem \eqref{eq:OCP_kell} with any suitable optimization algorithm. Once $\us_{0,L}$ is computed, $\QQ^L\LQ p_0(\us_{0,L})\RQ$ is obtained through a post-processing step by evaluating the forward and adjoint solution operators on $\us_{0,L}$ for every quadrature node involved by $\QQ^{L}$. Notice that this post-processing step can be avoided using a full-space optimization algorithm for \eqref{eq:full_space_system_discretized}, since $\QQ^L\LQ p_0(\us_{0,L})\RQ$ is already computed along the solution process and can be returned as an output of the optimization algorithm.
Next, for $\ell=1$, we consider a coarser quadrature formula of level $L-1$ and solve the minimization problem on both a mesh of level $1$ (thus finer) and of level $0$. The quantity $\QQ^{L-1}\LQ p_{1}(\us_{1,L-1})-p_{0}(\us_{0,L-1})\RQ$ is again either obtained as the output of a full-space optimization algorithm, or computed in a post-processing step. 
The process is then repeated for all remaining levels $\ell$ smaller or equal than $L$.
Given this algorithmic description, we stress that $\umq{L}$ does not correspond to the output of a \textit{new} optimization algorithm, but in constrast it is defined as the combination (possibly with a projection onto the set of admissible controls if $\Uad\neq U$) of outputs of \textit{any} optimization algorithm applied to a sequence of minimization problems. From this point of view, our multilevel approximation is agnostic with respect to the optimization algorithm used, and should be interpreted as a smart way of combining spatial and quadrature discretizations at different accuracy levels, instead of solving a single, but very expensive, optimization problem.
 \begin{algorithm}[]
\setlength{\columnwidth}{\linewidth}
\caption{Multilevel Quadrature Approximation}
\begin{algorithmic}[1]\label{alg:MLQ}
\item \textbf{Require}: Level $L$, a sequence of finite element spaces $\left\{V_\ell\right\}_{\ell=0}^L$, and quadratures $\left\{\QQ^k\right\}_{k=0}^L$.
\item Compute $\us_{0,L}$ by solving \eqref{eq:OCP_kell}.
\item Compute $\QQ^{L}\LQ p_0(\us_{0,L})\RQ$.
\item \textbf{For} $\ell=1,\dots, L$
\item Compute $\us_{\ell-1,L-\ell}$ by solving \eqref{eq:OCP_kell}.
\item Compute $\us_{\ell,L-\ell}$ by solving \eqref{eq:OCP_kell}.
\item Compute $\QQ^{L-\ell}\LQ p_\ell(\us_{\ell,L-\ell})\RQ$ and $\QQ^{L-\ell}\LQ p_{\ell-1}(\us_{\ell-1,L-\ell})\RQ$.
\item \textbf{Endfor}
\item \textbf{Output}: Return $u^{MQ(L)}$ using \eqref{eq:multilevel_approx}.
\end{algorithmic}
\end{algorithm}
Notice further that Algorithm \ref{alg:MLQ} could be potentially parallelized straightforwardly, since all minimization problems (lines 2,4,5) are completely independent. However, solving them in a sequential way has the advantage that $\us_{\ell-1,L-\ell}$ represents a good initialization to compute $\us_{\ell,L-\ell}$ (line 5), (and similarly $\us_{\ell,L-\ell}$ is a good initial guess for the next iteration of the for loop), thus, potentially, very few iterations of an outer nonlinear optimization algorithm are needed to solve the problems in lines 4-5.

\begin{remark}[On the use of two different quadrature formulae on each level]
On each level one could potentially use a quadrature formula (say, $\QQ_1^{L-\ell}$) to semidiscretize the OCP and a different one (say, $\QQ_2^{L-\ell}$) to compute the average of the resulting adjoint variable.
From the theoretical point of view, the error bounds \eqref{eq:lemma_error_statement} and \eqref{eq:err_bound_deterministic} would still hold true with $\QQ^{\ell}$ replaced by $\QQ_2^{\ell}$. The use of two different, randomized and statistically independent, quadrature formulae on each level would have a particular theoretical interest since it would greatly simplify the analysis presented in Section \ref{sec:convergence_analysis} (see also Remark \ref{remark:power_p}).
Nevertheless, it would prevent from recycling the adjoint variables that are computed along the optimization procedure. For this reason, in the rest of the manuscript we assume to use the same quadrature formulae, and detail the convergence analysis in this more delicate case.
\end{remark}

\section{Convergence analysis}\label{sec:convergence_analysis}
In this section, we formulate mild assumptions on the spatial and on the quadrature approximations in order to control the three terms appearing in \eqref{eq:lemma_error_statement}. To bound the last term, our analysis requires to make the additional assumption $\Uad\equiv U$, that is, our theory does not cover the case where control constraints are present.
We still however cover boundary or local controls, modeled by the control operator $\B$, and local observations through the operator $\C$.
We will numerically investigate the efficiency of \eqref{eq:multilevel_approx} in the presence of control constraints in the numerical Section \ref{sec:numerical_section}.

We start formulating an assumption on the sequence of quadrature formulae.

\begin{assumption}[Quadrature formula]\label{Ass:quad_form}
The quadrature formulae $\left\{\QQ^k\right\}_{k\geq 0}$ are random, unbiased, and statistically independent. Further, for any $q\in [2,\infty)$ there is a normed vector space of functions denoted by $H_q(\Omega;Z)$ and a constant $\widetilde{C}_{\QQ,q}>0$ such that for every $f\in H_q(\Omega;Z)$ and $k\geq 0$,
\begin{align}\label{ass_quad_formula}
\left\|\E\LQ f\RQ-\QQ^k \LQ f\RQ\right\|_{L^q(\Omega;Z)}:= \E\LQ \left\|\E\LQ f\RQ-\QQ^k \LQ f\RQ\right\|^q_Z\RQ^{\frac{1}{q}} \leq \widetilde{C}_{\QQ,q} \gamma_k \|f\|_{H_q(\Omega;Z)},
\end{align}
where $\left\{\gamma_k\right\}_{k\geq 0}$ is a decreasing null sequence.
\end{assumption}
The outer expectation operator in \eqref{ass_quad_formula} must be interpreted as an expectation over the randomness of the quadrature formula, while the inner expectation is taken with respect to underlying probability measure $\mathbb{P}$ in \eqref{eq:random_PDE_weak} and \eqref{eq:OCP_model_problem}.
Note that the Monte Carlo method satisfies \eqref{ass_quad_formula} with $\gamma_k=N_k^{-1/2}$ and $H_q(\Omega;Z)=L^q(\Omega;Z)$, where for $q>2$ the proof is based on the so-called Rademacher sequences, see, e.g., \cite{ledoux2013probability}. We refer to Remark \ref{remark:power_p} for a technical discussion on why Assumption \ref{Ass:quad_form} is formulated for a generic exponent $q\geq 2$.

A deterministic quadrature formula can as well be embedded into \eqref{ass_quad_formula}. In this case, we can simplify the assumption, postulating that the existance of a normed vector space of functions $H(\Omega;Z)$, equipped with a norm $\|\cdot\|_{H(\Omega;Z)}$, such that if $f\in H(\Omega;Z)$ then
\begin{equation}\label{ass_quad_formula_det}
\left\|\E\LQ f\RQ-\QQ^k \LQ f\RQ\right\|_Z \leq \widetilde{C}_{\QQ} \gamma_k \|f\|_{H(\Omega;Z)},
\end{equation}
where $\widetilde{C}_{\QQ}>0$.
As concrete examples, if the probability space can be parametrized by a vector $\xib$ of $M$ random variables assuming values on a bounded subset $\Gamma$ of $\setR^M$, Quasi-Monte Carlo methods satisfy \eqref{ass_quad_formula_det} where $\gamma_k =N_k^{-1} (\log N_k)^M$ and $H(\Omega;Z)$ denotes the space of functions with bounded Hardy-Krause variation \cite{owen2005multidimensional}. Alternatively, if $H(\Omega;Z)$ denotes the space of functions that admit an analytic extension on an open subset $\Sigma\subset \mathbb{C}^M$ containing $\Gamma$, we may obtain the exponential convergence rate $\gamma_k=\exp(-b N_{k}^{1/M})$, with norm $\|\cdot\|_{C^0(\Sigma;Z)}$, using tensorized Gaussian quadrature formulae.

For the linear-quadratic OCP \eqref{eq:OCP_model_problem}, standard calculations (see, e.g., \cite{martin2021complexity,doi:10.1137/19M1294952}) show that, for a given set of quadrature points, it holds that
\begin{equation}\label{eq:bound_quadrature_given_set}
\|\us-\us_{\infty,k}\|_U\leq C \|\E\LQ p(\us)\RQ -\QQ^{k}\LQ p(\us)\RQ\|_Z,
\end{equation}
where the constant $C=C(\nu,\mathcal{B})$ depends only on the regularization parameter and on the control operator.
By taking first the $q$-th power of \eqref{eq:bound_quadrature_given_set} and then the expectation with respect to the randomness of the quadrature, we conclude that under Assumption \ref{Ass:quad_form} it holds that
\begin{equation}\label{eq:bound_quadrature_error}
\|\us-\us_{\infty,k}\|_{L^q(\Omega;U)} \leq C_{\QQ,q} \gamma_k \|p(\us)\|_{H_q(\Omega;Z)},
\end{equation}
for a constant $C_{\QQ,q}>0$ provided that $p(\us)\in H_q(\Omega;Z)$. Equation \eqref{eq:bound_quadrature_error} is a bound on the error on the control due to the sole discretization in probability. Similarly, if $p_\ell(\us)\in H_q(\Omega;Z)$ the same arguments lead to 
\begin{equation}\label{eq:bound_quadrature_error_semidiscretized}
\|\us_{\ell,\infty}-\us_{\ell,k}\|_{L^q(\Omega;U)} \leq C_{\QQ,q} \gamma_k \|p_\ell(\us)\|_{H_q(\Omega;Z)}.
\end{equation}
Corresponding results for deterministic quadrature formulae hold true in the $\|\cdot\|_U$ norm using directly \eqref{ass_quad_formula_det} in \eqref{eq:bound_quadrature_given_set}.

Next, we formulate an assumption concerning the spatial approximation. In order to postulate some rate of convergence of the finite element discretizations, we suppose that the state and adjoint variables are more regular than general elements of $V$. Besides requiring that $\B$ defines a functional on a Hilbert space $Z$ containing $V$, we further assume that $\text{Im}\C^\star\subset W^\prime$, i.e., $\C^\star:H^\prime\rightarrow W^\prime$, for some Hilbert space $W$ such that $W\supsetneq V$.

\begin{assumption}[Finite element approximations]\label{Ass:fem}
There are subspaces $K$ and $\widetilde{K}$ of $V$ such that $\text{Im}(S^\w \B)\subset K$ and $\text{Im}(S^{\w,\star}\C^\star \Lambda_H) \subset \widetilde{K}$, and these operators are continuous from $U$ to $K$ and from $H$ to $\widetilde{K}$, respectively. Further,
the operators $S^\w_\ell$ and $S^{\w,\star}_\ell$ satisfy for $\PP\text{-a.e. } \w$,
\begin{align}\|(S^\w-S^\w_\ell)\B u\|_W &\leq C_a(\w) h_{\ell}^{\alpha_W}\|u\|_U,\quad \forall u\in U,\label{ass:forward}\\
\|(S^{\w,\star}-S^{\w,\star}_\ell)\C^\star \Lambda_{H}h\|_Z&\leq  C^\star_a(\w) h_{\ell}^{\alpha_Z}\|h\|_H,\quad \forall h\in H,\label{ass:adjoint}
\end{align}
for given $\alpha_W,\;\alpha_Z >0$ and positive constants $C_a(\w)$ and $C^\star_a(\w)$, such that $C_a,C_a^\star\in L^q(\Omega;\mathbb{R})$ for any $q\in [1,\infty)$.
\end{assumption}

\begin{example}
Consider the linear-quadratic OCP 
\[
\min_{u\in L^2(\D) ,y\in L^2(\Omega;H^1_0(\D))} \widehat{J}(y,u):= \frac{1}{2}\E\LQ \int_\D (y(\mathbf{x},\w)-y_d(\mathbf{x}))^2\;d\mathbf{x}\RQ +\frac{\nu}{2}\int_D u^2(\mathbf{x})\;d\mathbf{x},
\]
with distributed control and distributed observations, and subjected to a random linear elliptic PDE with homogeneous Dirichlet boundary conditions whose weak form, for $\PP$-a.e. $\w$, is: Find $y(\cdot,\w)\in H^1_0(\D)$ such that
\[\int_{\D} \kappa(\mathbf{x},\w)\nabla y(\mathbf{x},\w)\cdot\nabla v(\mathbf{x})\; d\mathbf{x}=\int_{\D} u(\mathbf{x})v(\mathbf{x})\; d\mathbf{x},\quad \forall v\in H^1_0(\D),\]
where $\kappa$ is a random diffusion coefficient. In this setting we have $U=L^2(\D)$, $V=H_0^1(\D)$, $Z=L^2(\D)$, $H=L^2(\D)$ and $W=L^2(\D)$. $\B$ is the embedding operator from $L^2(\D)$ to $(H^1_0(\D))^{-1}$, while $\C$ is the identity operator on $L^2(\D)$.
The Lipschitzianity of $\kappa$ and a $C^{1,1}$ domain are sufficient to have $H^2$-regularity of the state and adjoint variables (i.e., $K=\widetilde{K}=H^2(\D)$), which permits to satisfy \eqref{ass:forward} and \eqref{ass:adjoint} with $\alpha_W=\alpha_Z=2$ using continuous piecewise linear finite elements.
\end{example}
Assumption \ref{Ass:fem} is standard in finite element approximation analysis and the condition on the integrability on $C_a$ and $C_a^\star$ is sufficiently general to cover also log-normal fields \cite{charrier2012strong}.
We now recall some technical results that follow from Assumption \ref{Ass:fem} and that are relevant to study the error associated to $\umq{L}$. To simplify the notation, we set $C_\C:=\|\C\|_{\mathcal{L}(W,H)}$ and denoted the operator $\C^\star \Lambda_H \C:W\rightarrow W^\prime$ by $T$ and its continuity constant by $C_T$.
\begin{lemma}\label{lemma:pointwise}
Let Assumption \ref{Ass:fem} holds and set $\alpha:=\min(\alpha_W,\alpha_Z)$. Then,
\begin{align}
&\|(S^{\w,\star}\C^\star\Lambda_H\C S^{\w}\B-S_\ell^{\w,\star}\C^\star\Lambda_H\C S_\ell^{\w}\B) u\|_Z \leq A(\w) h_{\ell}^{\alpha}\|u\|_U,\quad \PP\text{-a.e. in }\Omega,\;\forall u\in U,\label{eq:difference_operators}\\
&\|p^\w(u)-p_\ell^\w(u)\|_Z \leq B(\w)  h_\ell^{\alpha}\left(\|u\|_U+\|y_d\|_H\right),\quad \PP\text{-a.e. in }\Omega,\;\forall u\in U,\;\forall y_d \in H,\label{eq:convergence_spatial_adjoint_same_u}
\end{align}
where
\begin{align*}
A(\w)&:=C_{S}(\w)\left(C_a^\star(\w)C_{\mathcal{C}}C_{\B}+C_{T}C_a(\w)\right),\qquad B(\w):=C_a^\star(\w)+ A(\w),
\end{align*}
and $A,B\in L^q(\Omega;\setR)$ for any $q\in [1,\infty)$.
Further, there exists a positive constant $C_d:=\frac{C_{\B}\|B\|_{L^2(\Omega)}}{\nu}$ and a random variable $G(\w):=\left(C_a^\star(\w)+C_S^2(\w)C_{T}C_{\B}C_d+A(\w)\right)$, $G\in L^q(\Omega;\setR)$  $\forall\;q \in [1,\infty)$, such that
\begin{align}
&\|\us-\us_{\ell,\infty}\|_U  \leq C_d h_{\ell}^{\alpha}\left(\|\us\|_U+\|y_d\|_H\right),\label{ass:mesh_ref_control}\\
&\|p^\w(\us)-p^\w_\ell(\us_{\ell,\infty})\|_Z  \leq G(\w) h_{\ell}^{\alpha} \left( \|\us\|_U+\|y_d\|_H\right),\quad \PP\text{-a.e. in }\Omega.\label{eq:convergence_spatial_adjoint}
\end{align}
\end{lemma}

\begin{proof}
The proof of the first inequality relies on the triangle inequality, the continuity of $S^{\w}$, $S^{\w,\star}_{\ell}$, $T$, $\mathcal{C}$ and Assumption \eqref{Ass:fem},
\begin{align*}
\|S^{\w,\star}\C^\star\Lambda_H\C S^{\w}\B u-S_\ell^{\w,\star}\C^\star\Lambda_H\C S_\ell^{\w}\B u\|_Z &=\|(S^{\w,\star}-S^{\w,\star}_\ell)\C^\star\Lambda_H\C S^{\w}\B u\|_Z\\
&+\|S_\ell^{\w,\star}\C^\star\Lambda_H\C (S_\ell^{\w}-S^{\w})\B u\|_Z\nonumber\\
&\leq C_a^\star(\w)\|\C\|_{\mathcal{L}(V,H)}C_S(\w)\|\B\|_{\mathcal{L}(U,V^\prime)} h_\ell^{\alpha_Z}\|u\|_U\\
& + \|S_\ell^{\w,\star}\|_{\mathcal{L}(W^\prime,Z)}C_{T}C_a(\w) h_\ell^{\alpha_W}\|u\|_U\\
&\leq C_a^\star(\w)C_{\C}C_S(\w)C_{\B}h_{\ell}^{\alpha_z}\|u\|_U\\
&+C_S(\w)C_TC_a(\w)h_{\ell}^{\alpha_W}\|u\|_U.
\end{align*}
The second inequality is a direct consequence since 
\begin{align*}
\|p^\w(u)-p_\ell^\w(u)\|_Z&=\|S^{\w,\star}\C^\star\Lambda_H (y_d-\C S^{\w}\B u)-S_\ell^{\w,\star}\C^\star\Lambda_H (y_d-\C S_\ell^{\w}\B u)\|_Z\\
&\leq \|(S^{\w,\star}-S^{\w,\star}_{\ell})(\C^\star\Lambda_Hy_d)\|_Z+\|S^{\w,\star}\C^\star\Lambda_H\C S^{\w}\B u-S_\ell^{\w,\star}\C^\star\Lambda_H\C S_\ell^{\w}\B u\|_V\\
&\leq C_a^\star(\w) h_{\ell}^{\alpha_Z}\|y_d\|_H+ A(\w) h_{\ell}^{\alpha}\|u\|_U.
\end{align*}
The integrability of $A,B$ follows directly from the assumptions on $C_S,C_a$ and $C^\star_a$.
To prove \eqref{ass:mesh_ref_control} we consider the two optimality conditions satisfied by $\us$ and $\us_{\ell,\infty}$, namely
\[\langle J^\prime(\us),v-\us\rangle_{U',U} \geq 0,\qquad \langle J^\prime_{\ell,\infty}(\us_{\ell,\infty}),v-\us_{\ell,\infty}\rangle_{U',U} \geq 0,\qquad \forall v\in \Uad.\]
Choosing $v=\us_{\ell,\infty}$ and $v=\us$ in the two inequalities respectively, summing them up and adding and subtracting $\langle J^\prime(\us_{\ell,\infty}),\us-\us_{\ell,\infty}\rangle_{U',U}$, standard arguments (see, e.g, \cite[Chapter 3]{hinze2008optimization}) lead to the bound
\[\nu\|\us-\us_{\ell,\infty}\|_U^2\leq \E\LQ (\mathcal{B}^\star \left(p^\w_\ell(\us)-p^\w(\us)\right),\us-\us_{\ell,\infty})_U\RQ.\]
Young's inequality and \eqref{eq:convergence_spatial_adjoint_same_u} finally yield
\[\|\us-\us_{\ell,\infty}\|_U^2\leq \frac{C_{\B}^2}{\nu^2}\E\LQ \|p^\w(\us)-p_\ell^\w(\us)\|_Z^2\RQ\leq  \frac{C_{\B}^2\E\LQ B^2(\w)\RQ}{\nu^2}h_{\ell}^{2\alpha}\left(\|\us\|_U+\|y_d\|_H\right)^2,
\]
from which the third claim follows. 
To prove the last inequality we observe that  
\begin{align*}
\|p^\w(\us)-p_\ell^\w(\us_{\ell,\infty})\|_Z&=\|S^{\w,\star}\C^\star\Lambda_H (y_d-\C S^{\w}\B \us)-S_\ell^{\w,\star}\C^\star\Lambda_H (y_d-\C S_\ell^{\w}\B \us_{\ell,\infty})\|_Z\\
&\leq C_a^\star(\w)h_{\ell}^{\alpha_Z}\|y_d\|_H +  \|S^{\w,\star}\C^\star\Lambda_H\C S^{\w}\B \us-S_\ell^{\w,\star}\C^\star\Lambda_H\C S_\ell^{\w}\B \us_{\ell,\infty}\|_Z,
\end{align*}
and then we bound the second term using \eqref{ass:mesh_ref_control},
\begin{align*}
\|S^{\w,\star}\C^\star\Lambda_H\C S^{\w}\B \us-S_\ell^{\w,\star}\C^\star\Lambda_H\C S_\ell^{\w}\B \us_{\ell,\infty}\|_Z &\leq 
\|S_\ell^{\w,\star}\C^\star\Lambda_H\C S_\ell^{\w}\B( \us-\us_{\ell,\infty})\|_Z\\
&+\|(S^{\w,\star}\C^\star\Lambda_H\C S^{\w} -S_\ell^{\w,\star}\C^\star\Lambda_H\C S_\ell^{\w})\B \us\|_Z\\
& \leq \left(C^2_S(\w)C_TC_{\B}C_d+ A(\w)\right) h_\ell^{\alpha}\left(\|\us\|_U+\|y_d\|_H\right).
\end{align*}
\end{proof}
Equation \eqref{ass:mesh_ref_control} provides a bound on the first term of \eqref{eq:lemma_error_statement}. It states that the solution of the semi-discrete optimal control converges to the full continuous one with a rate $\alpha$ determined by the approximation properties of the discrete state and adjoint solution operators in the $\|\cdot\|_W$ and $\|\cdot\|_Z$ norms.  Together, \eqref{eq:bound_quadrature_error_semidiscretized} and \eqref{ass:mesh_ref_control} lead to a bound for the error associated to the solution $\us_{\ell,k}$ of \eqref{eq:OCP_kell}. Specifically it holds that
\begin{equation}\label{eq:error_boundSL}
\E\LQ \|\us-\us_{k,\ell}\|_U^2\RQ \leq 2\|\us-\us_{\ell,\infty}\|_U^2 +2 \E\LQ \|\us_{\ell,\infty}-\us_{\ell,k}\|_U^2\RQ \leq 2C_1^2 h_{\ell}^{2\alpha} + 2C_2^2 \gamma^2_k,
\end{equation}
for two positive constant $C_1$ and $C_2$.

To analyze the error associated to the multilevel approximation $\umq{L}$ we need a further mixed-regularity assumption. In particular, we 
require \eqref{eq:convergence_spatial_adjoint} to hold not only pointwise for $\PP$-a.e. $\omega$, but also in the norm of the vector space of functions appearing in the assumption on the quadrature formula.
Notice that for Monte Carlo methods (for which we can take $H_q(\Omega;Z)=L^q(\Omega;Z)$, $q\in [2,\infty)$), 
this is trivially satisfied if $G\in H_q(\Omega;\setR)$. Similarly, if the solution operator and its adjoint (both at the continuous and discrete level) are holomorphic in a region $\Sigma \supset \Gamma$, it is sufficient to require $G\in C^0(\Sigma;\setR)$. Other quadrature rules, such as Quasi-Monte Carlo methods, may lead to $H(\Omega;Z)$-norms that involve the derivatives of the solution operators with respect to the random parameters, preventing such a straight characterization.
\begin{assumption}[Mixed-regularity]\label{ass:mixed_regularity}
There exists a $\qbar\in (2,\infty)$ such that for every $u\in U$ and $q\in [2,\qbar]$, $p(u)$ and $p_{\ell}(u)$ belong to $H_q(\Omega;Z)$.
Further, there exist a positive rate $\widetilde{\alpha}$ and a positive constant $C_{H,q}$ such that
\begin{equation}\label{ass:mixed}
\|p(\us)-p_{\ell}(\us_{\ell,\infty})\|_{H_q(\Omega;Z)} \leq C_{H,q} h_{\ell}^{\widetilde{\alpha}}\left( \|\us\|_U+\|y_d\|_H\right).
\end{equation}
\end{assumption}
In the case of a deterministic quadrature formula, we assume \eqref{ass:mixed} holds true in the corresponding $H(\Omega;Z)$ norm.

We are now ready to formulate the next Lemma which is key to our analysis. It will permit to control the third term of the error estimate \eqref{eq:lemma_error_statement}.
\begin{lemma}\label{lemma:surplus}
Let Assumptions \ref{Ass:quad_form}, \ref{Ass:fem} and \ref{ass:mixed_regularity} hold, set $\beta=\min\left(\alpha,\widetilde{\alpha}\right)$, and assume further that $\Uad=U$. Then for any $q\in [2,\qbar)$, 
there exist a constant $C=C(q,\qbar,\us,y_d)$ such that
\begin{equation}\label{eq:lemma}
\E\LQ \|\us-\us_{\ell,\infty}-\us_{\infty,k}+\us_{\ell,k}\|^q_U\RQ \leq C   h_{\ell}^{q\beta}\gamma^q_k,
\end{equation}
for every $\ell,k\geq 0$.
\end{lemma}
\begin{proof}
To simplify the notation, from now on we set $\Delta u:= \us-\us_{\ell,\infty}-\us_{\infty,k} +\us_{\ell,k}$.
We recall the optimality conditions satisfied by the different optimal controls appearing in \eqref{eq:lemma}, namely
\begin{align}
\langle J^\prime(\us),w-\us\rangle_{U^\prime,U} &= 0,\quad \forall w\in U,\label{eq:var_ini}\\
\langle J^\prime_{\ell,\infty}(\us_{\ell,\infty}),w-\us_{\ell,\infty}\rangle_{U^\prime,U} &= 0,\quad \forall w\in U,\label{eq:var_ini_ell}\\
\langle J^\prime_{\infty,k}(\us_{\infty,k}),w-\us_{\infty,k}\rangle_{U^\prime,U} &= 0,\quad \forall w\in U,\label{eq:var_ini_k}\\
\langle J^\prime_{\ell,k}(\us_{\ell,k}),w-\us_{\ell,k}\rangle_{U^\prime,U} &= 0,\quad \forall w\in U.\label{eq:var_ini_kell}
\end{align}
By choosing a different test function $w$ in each relation such that the second argument of the duality pairings equals $\Delta u$, and by further summing the four conditions we obtain
\[\langle -J^\prime(\us) +J_{\ell,\infty}^\prime(\us_{\ell,\infty}) + J_{\infty,k}^\prime(\us_{\infty,k}) -J_{\ell,k}^\prime(\us_{\ell,k}), \Delta u\rangle_{U^\prime,U} = 0.\]   
Then,
\begin{equation}
\begin{aligned}
\nu \| \Delta u\|^2_U
&= \langle \nu \Lambda_U \Delta u -J^\prime(\us) +J_{\ell,\infty}^\prime(\us_{\ell,\infty}) + J^\prime_{\infty,k}(\us_{\infty,k}) -J^\prime_{\ell,k}(\us_{\ell,k}), \Delta u\rangle_{U^\prime,U}\\
&=\langle -J^\prime(\us) +J_{\ell,\infty}^\prime(\us_{\ell,\infty}) +J^\prime_{\infty,k}(\us)-J^\prime_{\ell,k}(\us_{\ell,\infty}),\Delta u\rangle_{U^\prime,U} \\
&+\langle \nu\Lambda_U \Delta u -J^\prime_{\infty,k}(\us)+J^\prime_{\ell,k}(\us_{\ell,\infty})+ J^\prime_{\infty,k}(\us_{\infty,k}) -J_{\ell,k}^\prime(\us_{\ell,k}), \Delta u\rangle_{U^\prime,U}\\
&=I+II.
\end{aligned}
\end{equation}
Next, on the one hand, using the expressions of the gradients and the continuity of $\B$,
\begin{equation}\label{eq:lemma_final_first_term}
\begin{aligned}
I &=\langle \B^\star \left(\E\LQ p(\us)-p_\ell(\us_{\ell,\infty})\RQ-\QQ^k \LQ p(\us)-p_{\ell}(\us_{\ell,\infty})\RQ\right),\Delta u\rangle_{U^\prime,U}\\
&\leq C_{\B} \|\left( \E-\QQ^k\right)\LQ p(\us)-p_\ell(\us_{\ell,\infty})\RQ\|_Z \|\Delta u\|_U\\
&\leq \frac{C^2_{\B}}{\nu}\|\left( \E-\QQ^k\right)\LQ p(\us)-p_\ell(\us_{\ell,\infty})\RQ\|^2_Z+ \frac{\nu}{4}\|\Delta u\|_{U}^2.
\end{aligned}
\end{equation}
On the other hand, adding and subtracting $p_\ell(\us)-p_\ell(\us_{\infty,k})$
\begin{equation}
\begin{aligned}
II &=\langle \B^\star \QQ^k\LQ p(\us) -p_\ell(\us_{\ell,\infty})-p(\us_{\infty,k})+p_\ell(\us_{\ell,k})\RQ, \Delta u \rangle_{U^\prime,U}\\
&= \langle  \QQ^k\LQ p(\us)-p(\us_{\infty,k}) - p_\ell(\us)+ p_\ell(\us_{\infty,k})  \RQ, \B\Delta u \rangle_{Z,Z^\prime}\\
&+\langle \QQ^k\LQ  p_\ell(\us)- p_\ell(\us_{\infty,k})  -p_\ell(\us_{\ell,\infty})+p_\ell(\us_{\ell,k})
  \RQ,\B\Delta u \rangle_{Z,Z^\prime}.
  \end{aligned}
\end{equation}  
Using the affinity of the map $u\rightarrow  p_\ell(u)$ and the positivity of the weights of $\QQ^k\LQ\cdot\RQ$, we observe that
\begin{equation*}
\begin{aligned}
\langle \QQ^k\LQ  p_\ell(\us)- p_\ell(\us_{\infty,k})  -p_\ell(\us_{\ell,\infty})+p_\ell(\us_{\ell,k})
  \RQ,\B\Delta u \rangle_{Z,Z^\prime}&=\QQ^k \LQ \langle  -S^{\w,\star}_{\ell}\C^\star\Lambda_H\C S^\w_{\ell} \B \Delta u,\B\Delta u \rangle_{Z,Z^\prime}\RQ\\
  & =\QQ^k \LQ \langle  -\Lambda_H\C S^\w_\ell \B \Delta u,\C S^\w_\ell\B\Delta u \rangle_{H^\prime,H} \RQ \\
 &= \QQ^k \LQ -\|\C S^\w_\ell \B \Delta u\|_H^2 \RQ \\
  &\leq 0.
\end{aligned}
\end{equation*}
Hence, the second term is bounded using \eqref{eq:difference_operators},
\begin{equation}\label{eq:lemma_final_second_term}
\begin{aligned}
II &\leq \langle  \QQ^k\LQ p(\us)-p(\us_{\infty,k}) - p_\ell(\us)+ p_\ell(\us_{\infty,k})  \RQ, \B\Delta u \rangle_{Z,Z^\prime}\\
&\leq C_{\B} \QQ^k \|(S^{\w,\star}\C^\star\Lambda_H \C S^{\w}- S_{\ell}^{\w,\star}\C^\star\Lambda_H \C S_{\ell}^{\w})\mathcal{B}(\us-\us_{\infty,k})\|_Z \| \Delta u \|_U\\
&\leq C_{\B} \QQ^k \LQ A\RQ h_{\ell}^{\alpha}\|\us-\us_{\infty,k}\|_U  \| \Delta u \|_U\\
&\leq \frac{C^2_{\B} \left(\QQ^k\LQ A\RQ\right)^2}{\nu} h_{\ell}^{2\alpha}\|\us-\us_{\infty,k}\|^2_U  +\frac{\nu}{4}\| \Delta u \|^2_U.
\end{aligned}
\end{equation}
Setting $\widetilde{C}:=\frac{2C_\B^2(1+\QQ^k\LQ A\RQ)^2}{\nu^2}$,
\eqref{eq:lemma_final_first_term} and \eqref{eq:lemma_final_second_term} lead to
\begin{equation}\label{eq:intermediate_step}
\|\us-\us_{\ell,\infty}-\us_{\infty,k}+\us_{\ell,k}\|^2_U\leq \widetilde{C}\left(\left\|\left(\mathbb{E}-\QQ^{k}\right)\LQ p(\us)-p_{\ell}(\us_{\ell,\infty})\RQ\right\|_Z^2+h_{\ell}^{2\alpha}\|\us-\us_{\infty,k}\|_U^2\right),
\end{equation}
which holds true for any sample set of quadrature points, and $\widetilde{C}$ does depend on the sample set.
Taking the $q/2$-th power of both sides, $q\in [2,\qbar)$, in \eqref{eq:intermediate_step} yields
\begin{equation*}
\medmuskip=-1mu
\thinmuskip=-1mu
\thickmuskip=-1mu
\nulldelimiterspace=0.9pt
\scriptspace=0.9pt 
\arraycolsep0.9em 
\begin{aligned}
\|\us-\us_{\ell,\infty}-\us_{\infty,k}+\us_{\ell,k}\|^q_U & \leq \widetilde{C}^{\frac{q}{2}}\left(\left\|\left(\mathbb{E}-\QQ^{k}\right)\LQ p(\us)-p_{\ell}(\us_{\ell,\infty})\RQ\right\|_Z^2+h_{\ell}^{2\alpha}\|\us-\us_{\infty,k}\|_U^2\right)^{\frac{q}{2}}\\
&\leq 2^{\frac{q-2}{2}} \widetilde{C}^{\frac{q}{2}}  \left(\left\|\left(\mathbb{E}-\QQ^{k}\right)\LQ p(\us)-p_{\ell}(\us_{\ell,\infty})\RQ\right\|_Z^q+h_{\ell}^{q\alpha}\|\us-\us_{\infty,k}\|_U^q\right),
\end{aligned}
\end{equation*}
where the last step follows from Jensen's inequality being $\frac{q}{2}\geq 1$. We next take the expectation with respect to the sample set and use H\"{o}lder inequality, with $r=\frac{\qbar}{q}>1$ and $r^\prime=\frac{\qbar}{\qbar-q}$, Assumption \ref{Ass:quad_form} and \eqref{eq:bound_quadrature_error},
\begin{equation*}
\begin{aligned}
&\E \LQ \|\us-\us_{\ell,\infty}-\us_{\infty,k}+\us_{\ell,k}\|^q_U\RQ \leq 2^{\frac{q-2}{2}}\bigg(\E\LQ \widetilde{C}^{\frac{q}{2}}\left\|\left(\E-\QQ^{k}\right)\LQ p(\us)-p_\ell(\us_{\ell,\infty})\RQ\right\|^q_Z\RQ\\
&+h_{\ell}^{q\alpha}\E\LQ \widetilde{C}^{\frac{q}{2}} \|\us-\us_{\infty,k}\|^q_U\RQ\bigg)\\
&\leq  2^{\frac{q-2}{2}}\E\LQ \widetilde{C}^{\frac{q\qbar}{2(\qbar-q)}}\RQ^{\frac{\qbar-q}{\qbar}} \bigg(\widetilde{C}_{\QQ,\qbar}^q\gamma_{k}^{q}\| p(\us)-p_\ell(\us_{\ell,\infty})\|^q_{H_{\qbar}(\Omega;Z)} \\
&+h_{\ell}^{q\alpha} C_{\QQ,\qbar}^q \gamma_k^q\|p(\us)\|_{H_{\qbar}(\Omega;Z)}^q\bigg),\\
\end{aligned}
\end{equation*}
where we remark that $\E\LQ \widetilde{C}^{\frac{q\qbar}{2(\qbar-q)}}\RQ<\infty$ since $\frac{q\qbar}{2(\qbar-q)}>1$, $\QQ^k$ is unbiased, $A\in L^r(\Omega;\setR)$ for any $r\in [1,\infty)$ and, hence, using Jensen's inequality,
\[\E\LQ \left(\QQ^k\LQ A\RQ\right)^{\frac{q\qbar}{2(\qbar-q)}}\RQ\leq \E\LQ \QQ^k\LQ A^{\frac{q\qbar}{2(\qbar-q)}}\RQ\RQ=\E\LQ A^{\frac{q\qbar}{2(\qbar-q)}}\RQ<\infty.\]
Finally, Assumption \ref{Ass:quad_form}, , and \eqref{ass:mixed} lead to the claim
\begin{equation*}
\E \LQ \|\us-\us_{\ell,\infty}-\us_{\infty,k}+\us_{\ell,k}\|^q_U\RQ  \leq C \gamma_k^q h_{\ell}^{q\beta},
\end{equation*}
for a suitable constant $C=C(q,\qbar,\us,y_d)$.
\end{proof}

\begin{remark}[Lemma \ref{lemma:surplus} for deterministic quadrature formulae]
For deterministic quadrature formulae, the proof of Lemma \ref{lemma:surplus} is identical until \eqref{eq:intermediate_step}. We then take the square rooth, and it is then sufficient to ask that $\QQ^{k}\LQ A(\w)\RQ$ remains bounded for any $k\geq 0$ (e.g., $A(\w)$ is sufficiently regular so that the quadrature formula converges as $k$ increases), and to use \eqref{ass_quad_formula_det}, \eqref{eq:bound_quadrature_given_set} and \eqref{ass:mixed}.
The final bound reads
\begin{equation}\label{eq:lemma_deterministic}
\|\us-\us_{\ell,\infty}-\us_{\infty,k}+\us_{\ell,k}\|_U\leq C h_{\ell}^\beta\gamma_k,
\end{equation}
for a suitable constant $C=C(\us,y_d)$ and $\beta=\min(\alpha,\widetilde{\alpha})$.
\end{remark}

We are now ready to prove our main convergence result.
\begin{theorem}[Error of the multilevel approximation]\label{thm:error}
Let Assumptions \ref{Ass:quad_form}, \ref{Ass:fem} and \ref{ass:mixed_regularity} hold. Further, assume that $\Uad=U$. Then, the multilevel approximation \eqref{eq:multilevel_approx} satisfies the error bound
\begin{equation}\label{eq:complexity_result}
\E\LQ \|\us-\umq{L}\|^2_U\RQ \leq  2C_1^2 h_{L}^{2\alpha} +C_3^2L\sum_{\ell=0}^L h_{\ell}^{2\beta}\gamma^2_{L-\ell},
\end{equation}
for two positive constants $C_1$ and $C_3$.
\end{theorem}
\begin{proof}
We start by recalling Proposition \ref{lemma:error_bound} and bounding the first bias term thanks to \eqref{ass:mesh_ref_control},
\begin{equation}\label{eq:proof_complexity1}
\begin{aligned}
\E\LQ \|\us-\umq{L}\|^2_U\RQ &\leq 2C_d^2 h_{L}^{2\alpha}(\|\us\|_U+\|y_d\|_H)^2  \\
&+\frac{4C_{\B}}{\nu^2}\bigg(\sum_{\ell=0}^L\E\LQ \|\left(\mathbb{E}-\QQ^{L-\ell}\right)(p_{\ell}(\us_{\ell,\infty})-p_{\ell-1}(\us_{\ell-1,\infty})\|_Z^2\RQ \\
& + L\sum_{\ell=0}^L\E\LQ \QQ^{L-\ell}\|p_{\ell}(\us_{\ell,\infty})-p_{\ell}(\us_{\ell,L-\ell})-p_{\ell-1}(\us_{\ell-1,\infty})+p_{\ell-1}(\us_{\ell-1,L-\ell})\|_Z^2\RQ\bigg)
\end{aligned}
\end{equation}
The second term is controlled using Assumption \ref{Ass:quad_form} on the quadrature formula and Assumption \ref{ass:mixed_regularity} on the mixed-regularity, yielding
\begin{align*}
\medmuskip=-1mu
\thinmuskip=-1mu
\thickmuskip=-1mu
\nulldelimiterspace=0.9pt
\scriptspace=0.9pt 
\arraycolsep0.9em 
\sum_{\ell=0}^L \E\LQ \left\|\left(\E-\QQ^{L-\ell}\right) \LQ p_{\ell}(\us_{\ell,\infty})-p_{\ell-1}(\us_{\ell-1,\infty})\RQ\right\|^2_Z\RQ &\leq \widetilde{C}^2_{\QQ,2} \sum_{\ell=0}^L\gamma^2_{L-\ell} \|p_{\ell}(\us_{\ell,\infty})-p_{\ell-1}(\us_{\ell-1,\infty})\|^2_{H_2(\Omega;Z)}\nonumber \\
&\leq D \sum_{\ell=0}^L \gamma^2_{L-\ell}h_{\ell}^{2\widetilde{\alpha}}, 
\end{align*}
for a constant $D:=2\widetilde{C}^2_{\QQ,2}C_{H,2}^2 \left(\|\us\|_U+\|y_d\|_H\right)^2 (1+\tau^{2\widetilde{\alpha}})>0$, where $\tau=\min_{\ell} \frac{h_{\ell-1}}{h_{\ell}}$ is assumed to be finite.
Concerning the third term, we argue pointwise for every node $\xi^{L-\ell}_i$ involved by $\QQ^{L-\ell}$ and add and substract suitable terms. Dropping the superindex $L-\ell$ to ease notation,
\begin{align*}
&\|p^{\xi_i}_{\ell}(\us_{\ell,\infty})-p^{\xi_i}_{\ell-1}(\us_{\ell-1,\infty}) -p^{\xi_i}_{\ell}(\us_{\ell,L-\ell})+p^{\xi_i}_{\ell-1}(\us_{\ell-1,L-\ell})\|^2_Z\leq \\
&2\|p^{\xi_i}(\us)-p^{\xi_i}(\us_{\infty,L-\ell})-p^{\xi_i}_{\ell}(\us_{\ell,\infty})+p^{\xi_i}_{\ell}(\us_{\ell,L-\ell})\|^2_Z\\
&+2\|p^{\xi_i}(\us)-p^{\xi_i}(\us_{\infty,L-\ell})-p^{\xi_i}_{\ell-1}(\us_{\ell-1,\infty})+p^{\xi_i}_{\ell-1}(\us_{\ell-1,L-\ell})\|_Z^2.
\end{align*}
The two terms are similar, thus we consider only the first one. Adding and substracting $p^{\xi_i}_{\ell}(\us)-p^{\xi_i}_{\ell}(\us_{\infty,L-\ell})$ yields,
\begin{align*}
&\|p^{\xi_i}(\us)-p^{\xi_i}(\us_{\infty,L-\ell})-p^{\xi_i}_{\ell}(\us_{\ell,\infty})+p^{\xi_i}_{\ell}(\us_{\ell,L-\ell})\|^2_Z \\
&\leq 2\|p^{\xi_i}(\us)-p^{\xi_i}(\us_{\infty,L-\ell})-p^{\xi_i}_{\ell}(\us)+p^{\xi_i}_{\ell}(\us_{\infty,L-\ell})\|^2_Z\\
&+2\|p^{\xi_i}_{\ell}(\us)-p^{\xi_i}_{\ell}(\us_{\infty,L-\ell})-p^{\xi_i}_{\ell}(\us_{\ell,\infty})+p^{\xi_i}_{\ell}(\us_{\ell,L-\ell})\|_Z^2
\end{align*}
Then, on the one hand, using \eqref{eq:difference_operators} and H\"older inequality for $r=\frac{\qbar}{2}$ together with \eqref{eq:bound_quadrature_error} yields,
\begin{align}\label{eq:proof_complexity2}
&\E\LQ \QQ^{L-\ell}\left\|p(\us)-p(\us_{\infty,L-\ell})-p_\ell(\us)+p_\ell(\us_{\infty,L-\ell})\right\|^2_Z\RQ \nonumber\\
&=\E\LQ \QQ^{L-\ell}\left\|S^{\w,\star}TS^{\w}\B (\us-\us_{\infty,L-\ell})-S_{\ell}^{\w,\star}TS_{\ell}^{\w}\B(\us-\us_{\infty,L-\ell})\right\|^2_Z\RQ \nonumber\\
&\leq h_{\ell}^{2\alpha} \E\LQ \QQ^{L-\ell}\LQ A^2\RQ\|\us-\us_{\infty,L-\ell}\|^2_U\RQ \nonumber \\
&\leq h_\ell^{2\alpha} \left(\E\LQ \left(\QQ^{L-\ell}\LQ A^2\RQ^{\frac{\qbar}{\qbar-2}}\right) \RQ\right)^{\frac{\qbar-2}{\qbar}}\left(\E\LQ \|\us-\us_{\infty,L-\ell}\|^{\qbar}_U\RQ\right)^{\frac{2}{\qbar}}\nonumber\\
&\leq \widehat{C} h_\ell^{2\alpha} \gamma^2_{L-\ell},
\end{align}
where $\widehat{C}=\left(\E\LQ \left(\QQ^{L-\ell}\LQ A^2\RQ^{\frac{\qbar}{\qbar-2}}\right) \RQ\right)^{\frac{\qbar-2}{\qbar}}C_{\QQ,\qbar}^2\|p(\us)\|^2_{H_{\qbar}(\Omega;Z)}.$
On the other hand, using the affinity and continuity of the map $u\mapsto p_\ell(u)$, H\"older inequality for a $\delta$ such that $2(1+\delta)<\qbar$, and Lemma \ref{lemma:surplus},
\begin{align}
&\E \LQ \QQ^{L-\ell}\left\|p_\ell(\us)-p_\ell(\us_{\infty,L-\ell})-p_{\ell}(\us_{\ell,\infty})+p_{\ell}(\us_{\ell,L-\ell})\right\|^2_Z\RQ=\nonumber\\
&\E\LQ \QQ^{L-\ell}\left\|S^{\w,\star}_{\ell}\C^\star\Lambda_H\C S^{\w}_{\ell}(\us-\us_{\infty,L-\ell}-\us_{\ell,\infty}+\us_{\ell,L-\ell})\right\|^2_Z\RQ\nonumber\\
&\leq \E\LQ \QQ^{L-\ell}\LQ C_S^4C^2_T\RQ \|\us-\us_{\infty,L-\ell}-\us_{\ell,\infty}+\us_{\ell,L-\ell}\|^2_U\RQ\nonumber \\
&\leq \left(\E\LQ \QQ^{L-\ell}\LQ C_S^4C^2_T\RQ^{\frac{1+\delta}{\delta}}\RQ\right)^{\frac{\delta}{1+\delta}}\left(\E\LQ\|\us-\us_{\infty,L-\ell}-\us_{\ell,\infty}+\us_{\ell,L-\ell}\|^{2(1+\delta)}_U\RQ\right)^{\frac{1}{1+\delta}}\nonumber\\
&\leq \widehat{C}_2 h_{\ell}^{2\beta} \gamma^2_{L-\ell},\label{eq:proof_complexity3}
\end{align}
where $\widehat{C}_2:=\left(\E\LQ \QQ^{L-\ell}\LQ C_S^4C^2_T\RQ^{\frac{1+\delta}{\delta}}\RQ\right)^{\frac{\delta}{1+\delta}}C^{\frac{1}{1+\delta}}$, $C$ being the constant appearing in Lemma \ref{lemma:surplus}.
Recalling together \eqref{eq:proof_complexity1},\eqref{eq:proof_complexity2} and \eqref{eq:proof_complexity3}, we get the claim,
\begin{equation*}
\E\LQ \|\us-\umq{L}\|^2_U\RQ \leq 2C_1^2 h_{L}^{2\alpha} +C_3^2L\sum_{\ell=0}^L h_\ell^{2\beta}\gamma^2_{L-\ell},
\end{equation*}
for suitable costants $C_1$ and $C_3$.
\end{proof}

\begin{remark}[On Assumption \ref{Ass:quad_form}]\label{remark:power_p}
Assumption \ref{Ass:quad_form} on the control of the $L^q(\Omega;Z)$ norm of the quadrature error for any $q\in [2,\infty)$ is needed in Lemma \ref{lemma:surplus} and Theorem \ref{thm:error}. We remark that either using two random statistically independent quadrature formulae for the semidiscretization of the OCP and for quadrature of the adjoint variables, or assuming that $C_S, C_a$, and $C_a^\star$ are uniformly bounded, would permit to perform the analysis assuming only a bound on the $L^2(\Omega;Z)$-norm of the quadrature error.
\end{remark}
\begin{remark}[Theorem \ref{thm:error} for deterministic quadrature formulae]
The proof of Theorem \ref{thm:error} simplifies if the quadrature formulae are deterministic, since there are not outer expectations of products of sample-dependent quantities that have to be handled via H\"older inequality (as in \eqref{eq:proof_complexity2} and \eqref{eq:proof_complexity3}). Starting from the error bound \eqref{eq:err_bound_deterministic}, using \eqref{ass:mesh_ref_control} to bound the bias term, \eqref{ass_quad_formula_det} and \eqref{ass:mixed} to bound the second term, and \eqref{eq:difference_operators}, \eqref{eq:bound_quadrature_given_set} and \eqref{eq:lemma_deterministic} to control the third term leads to the final result
\begin{equation}\label{eq:complexity_result_deterministic}
\|\us-\umq{L}\|_U \leq  C_{D,1} h_{L}^{\alpha} +C_{D,3}\sum_{\ell=0}^L h_{\ell}^{\beta}\gamma_{L-\ell},
\end{equation}
for suitable positive constants $C_{D,1}$ and $C_{D,3}$.
\end{remark}
\begin{remark}[On control constraints]\label{remark:control_constraint}
The assumption $\Uad= U$ is used exclusively in the proof of Lemma \ref{lemma:surplus} to choose suitably the test functions $w$ in all relations \eqref{eq:var_ini}-\eqref{eq:var_ini_kell} and, in turn, Lemma \ref{lemma:surplus} permits to bound the expression in \eqref{eq:proof_complexity3}. Numerical experiments presented in Section \ref{sec:numerical_section} show that Lemma \ref{lemma:surplus} does not necessarily hold in the constrained case, but still the multilevel framework can be effectively applied to problems with box-constraints. From this perspective, our analysis could be sharpened by first refining the error bound of Proposition \ref{lemma:error_bound}, see the discussion in Section \ref{sec:numerical_section}.
\end{remark}

\section{Complexity analysis}\label{sec:Complexity}

In this section, we present a complexity analysis to compare the asymptotic cost of the multilevel approximation $\umq{L}$ and that of a standard, single level, approximation $\us_{\ell,k}$ to achieve a desired error tolerance $\varepsilon$.

Let $\left\{\QQ^k\right\}_{k\geq 0}$ be a family of randomized quadrature formulae satisfying Assumption \ref{Ass:quad_form} with $\gamma_k=N_k^{-\eta}$, for a $\eta\in (0,\infty)$ and $N_k$ denoting the number of quadrature points for every $k\geq 0$. Furthermore, let $h_\ell=2^{-\ell}$, $\mathrm{dim}V_{\ell}=2^{\ell d}$, and assume that the cost of computing $\us_{\ell,k}$ is proportional to $\mathrm{dim}(V_{\ell})N_k$. This is verified if, e.g., the outer nonlinear optimization algorithm converges in a number of iterations that is independent on the discretization parameters, and the cost of each iteration is proportional to $\mathrm{dim}(V_{\ell})N_k$ (see, e.g., \cite{vanzan,ciaramella2024multigrid} for optimal linear solvers.)
Then, recalling \eqref{eq:error_boundSL}, 
\begin{equation}\label{eq:error_SL_complexity}
\E \LQ \|\us-\us_{\ell,k}\|_U^2\RQ\leq 2C^2_1 h_{\ell}^{2\alpha}+2C^2_2 N_k^{-2\eta},
\end{equation}
and imposing that both contributions are smaller than $\frac{\varepsilon^2}{2}$ leads to the choices
\begin{equation}\label{eq:L_N_k}
L=\left\lceil\frac{1}{\alpha}\log_2\left(2C_1\varepsilon^{-1}\right)\right\rceil\quad\text{and}\quad N_K=\left\lceil (2C_2\varepsilon^{-1})^\frac{1}{\eta}\right\rceil, 
\end{equation}
being $\lceil\cdot\rceil$ the ceiling function, that is, $\lceil x\rceil$ is the unique integer satisfying $x\leq \lceil x\rceil <x+1$, $\forall x\in \setR$. Hence, to compute an approximation with an error tolerance $\varepsilon$, a single level approach requires a computational work $W_{\SL}$ of order
\begin{equation}\label{eq:asymptotic_cost_MC}
W_{\SL}(\varepsilon)\propto\mathrm{dim}V_{L}N_K\leq 2^{Ld} \left[(2C_2\varepsilon^{-1})^\frac{1}{\eta}+1\right] \sim \varepsilon^{-\frac{d}{\alpha}}+\varepsilon^{-\frac{d}{\alpha}-\frac{1}{\eta}}.
\end{equation}

Similarly, for $\umq{L}$ we wish to choose $L$ and the number of quadrature points on each level in order to meet the tolerance criterium. This analysis leads to the following result, which is a generalization of \cite[Theorem 3.1]{giles2008multilevel}. 
\begin{theorem}[Complexity result of the Multilevel Quadrature approximation]\label{thm:complexity}
Let $\left\{\QQ^k\right\}_{k\geq 0}$ be a family of randomized quadrature formulae satisfying Assumption \ref{Ass:quad_form} with $\gamma_k=N_k^{-\eta}$, $\eta\in (0,\infty)$.
Let $h_{\ell}=2^{-\ell}$, $\mathrm{dim}(V_\ell)=2^{d\ell}$, and assume that the cost of minimizing $J_{\ell,k}$ is proportional to $\mathrm{dim}(V_\ell)N_k$. Then, choosing
\begin{equation}\label{eq:L_ML}
L=\left\lceil\frac{1}{\alpha}\log_2\left(2C_1\varepsilon^{-1}\right)\right\rceil
\end{equation}
and
\begin{equation}\label{eq:complexity_ML2}
N_{\ell}=\left\lceil \left(\sqrt{2LC_3^2}\varepsilon^{-1}\right)^\frac{1}{\eta} h_{\ell}^{\frac{2\beta}{2\eta+1}}\mathrm{dim}V_\ell^{-\frac{1}{2\eta+1}}\left(\sum_{j=0}^L h_j^{\frac{2\beta}{2\eta+1}}\mathrm{dim}V_j^{\frac{2\eta}{2\eta+1}}\right)^{\frac{1}{2\eta}}\right\rceil
\end{equation}
for every $\ell\in\left\{0,\dots,L\right\}$, the multilevel approximation $\umq{L}$ achieves an error $\varepsilon\in (0,\frac{1}{2}]$ at the asymptotic computational cost,
\begin{equation*}
W_{ML}(\varepsilon)\sim 
\begin{cases}
\varepsilon^{-\frac{d}{\alpha}}+\varepsilon^{-\frac{1}{\eta}}|\log_2 \varepsilon^{-1}|^\frac{1}{2\eta}   \quad &\beta>\eta d,\\
\varepsilon^{-\frac{d}{\alpha}}+\varepsilon^{-\frac{1}{\eta}}|\log_2 \varepsilon^{-1}|^\frac{1}{\eta}\quad &\beta =\eta d,\\
\varepsilon^{-\frac{d}{\alpha}}+\varepsilon^{-\frac{1}{\eta}-\frac{2\beta-2\eta d}{2\eta\alpha} }|\log_2 \varepsilon^{-1}|^\frac{1}{2\eta} \quad &\beta <\eta d.\\
\end{cases}
\end{equation*}
\end{theorem}
\begin{proof}
Due to Theorem \ref{thm:error}, $\umq{L}$ satisfies, upon relabelling the sequence of quadrature formulae, the error bound
\begin{equation}\label{eq:error_bound_MLMC}
\E\LQ \|\us -\umq{L}\|_U^2\RQ \leq 2C_1^2 h_{L}^{2\alpha} + L C^2_3 \sum_{\ell=0}^L h_{\ell}^{2\beta}\gamma_{\ell}^{2},
\end{equation}
for the same constant $C_1$ appearing in \eqref{eq:error_SL_complexity}, and for $C_3>0$. The maximum level $L$ is chosen again so that the bias contribution is smaller than $\frac{\varepsilon^2}{2}$ leading to \eqref{eq:L_ML}. Note that if $\varepsilon\in (0,\frac{1}{2}]$ then
\begin{equation}\label{eq:bound_L}
L\leq \frac{1}{\alpha}\log_2\left(2C_1\varepsilon^{-1}\right)+1\leq \overline{C}_1|\log_2 \varepsilon^{-1}|,
\end{equation}
with $\overline{C}_1=\frac{1}{\alpha}\log_2(2 C_1)+\frac{1}{\alpha}+1$. In addition,
\begin{equation}\label{eq:bound_sumdimV_l}
\sum_{\ell=0}^L \mathrm{dim}V_\ell=\sum_{\ell=0}^L 2^{\ell d}=2^{L d}\sum_{\ell=0}^L 2^{-\ell d}\leq \overline{C}_2\varepsilon^{-\frac{d}{\alpha}},
\end{equation}
with $\overline{C}_2=\frac{2^{2d}}{2^d-1}(2C_1)^{\frac{d}{\alpha}}$.
The optimal number of quadrature points on each level is derived by minimizing the computational cost, constrained by the second error contribution to be smaller than $\frac{\epsilon^2}{2}$, that is,
\begin{equation}\label{eq:complexity_ML}
\begin{aligned}
\min_{\left\{N_{\ell}\right\}_{\ell=0}^L}& \sum_{\ell=0}^L \mathrm{dim}(V_{\ell}) N_\ell\\
\text{s.t.}&\quad L C_3^2 \sum_{\ell=0}^L h_{\ell}^{2\beta}N_{\ell}^{-2\eta}\leq \frac{\varepsilon^2}{2}. 
\end{aligned}
\end{equation}
By solving the constrained optimization problem using a Lagrangian approach, we obtain
\eqref{eq:complexity_ML2}. Next, we insert \eqref{eq:L_ML} and \eqref{eq:complexity_ML2} into the formula of the computational cost and, together with \eqref{eq:bound_sumdimV_l}, we obtain
\begin{equation}\label{eq:complexity_ML3}
W_{ML}(\varepsilon)\propto \sum_{\ell=0}^L \mathrm{dim}(V_\ell)N_{\ell}\leq \sum_{\ell=0}^L \mathrm{dim}(V_\ell)(N_{\ell}+1)
 \leq \overline{C}_2 \varepsilon^\frac{-d}{\alpha}+ (\sqrt{2LC_3^2}\varepsilon^{-1})^\frac{1}{\eta} \left(\sum_{\ell=0}^L 2^{-\frac{\ell (2\beta-2\eta d)}{2\eta+1}}\right)^{\frac{2\eta+1}{2\eta}}.
\end{equation}
Finally, the claim follows using standard arguments, see, e.g., \cite[Theorem 3.1]{giles2008multilevel}, to bound the summation based on the relative position of $\beta$ and $\eta d$, and using \eqref{eq:bound_L} to bound $L$ asymptotically with respect to $\varepsilon$.
\end{proof}

\begin{example}[A multilevel Monte Carlo Approximation]
If $\left\{\QQ^k\right\}_{k\geq 0}$ denotes a sequence of Monte Carlo approximations of increasing number of quadrature points, \eqref{eq:multilevel_approx} reduces to 
\begin{equation}\label{eq:MLMC}
\medmuskip=-1mu
\thinmuskip=-1mu
\thickmuskip=-1mu
\nulldelimiterspace=0.9pt
\scriptspace=0.9pt 
\arraycolsep0.9em 
\begin{aligned}
\umlmc{L}&:=\Pro\left(\frac{1}{\nu}\Lambda_U^{-1}\B^\star\sum_{\ell=0}^L \frac{1}{N_\ell}\sum_{i=1}^{N_\ell}\LQ p^{\w_{i,\ell}}_{\ell}(\usMLMC{\ell}_{\ell,\ell})-p^{\w_{i,\ell}}_{\ell-1}(\usMLMC{\ell-1}_{\ell-1,\ell})\RQ \right)\\
&=\Pro\left(\frac{1}{\nu}\Lambda_U^{-1}\B^\star \left\{ \frac{1}{N_0}\sum_{i=1}^{N_0} p^{\w_{i,0}}_{0}(\usMLMC{0}_{0,0}) + \sum_{\ell=1}^L \frac{1}{N_\ell} \sum_{i=1}^{N_\ell} \left(p^{\w_{i,\ell}}_{\ell}(\usMLMC{\ell}_{\ell,\ell})-p^{\w_{i,\ell}}_{\ell-1}(\usMLMC{\ell}_{\ell-1,\ell})\right)\right\}\right).
\end{aligned}
\end{equation}
where $\overrightarrow{\w}_{\ell}=\left\{\w^{i,\ell}\right\}_{i=1}^{N_\ell}$ are the independent, randomly drawn, samples.
Concerning its complexity, the Monte Carlo method satisfies the hypothesis of Theorem \ref{thm:complexity} with $\eta=\frac{1}{2}$. The optimal number of samples on each level then simplifies to
\begin{equation}\label{eq:N_ell_MLMC}
N_{\ell}=\left\lceil 2LC_3^2\varepsilon^{-2} h_{\ell}^{\beta}\sqrt{\mathrm{dim}V_\ell}\left(\sum_{j=0}^L h_j^{\beta}\sqrt{\mathrm{dim}V_j}\right)\right\rceil,
\end{equation}
which permits to achieve, together with $L$ as in \eqref{eq:L_ML}, a tolerance $\varepsilon\in (0,\frac{1}{2}]$ with a complexity
\begin{equation}\label{eq:complexity_MLMC}
W_{MLMC}(\varepsilon)\sim 
\begin{cases}
\varepsilon^{-\frac{d}{\alpha}}+\varepsilon^{-2}|\log_2 \varepsilon^{-1}|   \quad &2\beta< d,\\
\varepsilon^{-\frac{d}{\alpha}}+\varepsilon^{-2}|\log_2 \varepsilon^{-1}|^2\quad &2\beta = d,\\
\varepsilon^{-\frac{d}{\alpha}}+\varepsilon^{-2+\frac{2\beta- d}{\alpha} }|\log_2 \varepsilon^{-1}| \quad &2\beta > d.\\
\end{cases}
\end{equation}
Compared to \cite[Theorem 3.1]{giles2008multilevel}, \eqref{eq:complexity_MLMC} differs by an extra term $|\log (\varepsilon^{-1})|$ which is due to the factor $L$ multiplying the summation in \eqref{eq:error_bound_MLMC}.
\end{example}
\begin{remark}[On the dependence on the parametric dimension]
Assume that the probability space is parametrized by a vector of $M$ random variables.
Since the Monte Carlo method satisfies the assumptions of Theorem \ref{thm:complexity} with $\eta=\frac{1}{2}$, independent on $M$, the rates in \eqref{eq:complexity_MLMC} hold even for $M=\infty$ and thus \eqref{eq:MLMC} is an effective multilevel approximation even for OCP constrained by high-dimensional parametric PDEs. For other quadrature formulae, such as Quasi-Monte Carlo and stochastic collocation, the value of $\eta$ does in general depend on $M$, and so the complexity of the multilevel approximation depends, like that of a single-level approach, on the dimension $M$.
\end{remark}

\section{Numerical experiments}\label{sec:numerical_section}
In this section we verify the assertions of Lemma \ref{lemma:surplus} and Theorem \ref{thm:complexity}, explore numerically the performance of the multilevel approximation in the presence of box constraints, and compare the complexity of the two solution strategies described in Section \ref{sec:Complexity}. 

We consider the OCP \eqref{eq:OCP_model_problem} with distributed control and observations set on the domain $\mathcal{D}=(0,1)^d$, $d\in \left\{1,2\right\}$, and set $\nu=10^{-2}$. We numerically investigate both the unconstrained case ($\Uad=U$), and the constrained case with $\Uad=\left\{v \in L^2(\D):\; a\leq v(\mathbf{x})\leq b\; \forall \mathbf{x}\in \D\right\}$. 
As bilinear form we use
\[a_\w(y,v)=\int_{\D} \kappa(\mathbf{x},\w)\nabla y\cdot\nabla v\; d\mathbf{x},\]
$\kappa(\mathbf{x},\w)$ being the diffusion coefficient and equal to 
\[\kappa(\mathbf{x},\w)=\exp\left(\sigma^2\cdot (\xi_1(\w)\psi_1(\mathbf{x})+\xi_2(\w)\psi_2(\mathbf{x})+\xi_3(\w)\psi_3(\mathbf{x})
\right)).\]
The variables $\xi_n\sim \mathcal{U}(-1,1)$ are uniformly distributed and $\sigma^2=\exp(-1.125)$. We consider $M=3$ random variables so that an accurate reference solution can be computed using a stochastic collocation discretization of the expectation operator on anisotropic full-tensor grid and on a very fine computational mesh. 
For $d=1$, the functions $\psi_n$ are
\[\psi_1(x)=\cos(\pi x),\quad \psi_2(x)=\sin(\pi x),\quad \psi_3(x)=\cos(2 \pi x),\]
the target state is $y_d(x)=\exp(2x)\sin(2\pi x)$, and $a=-1$ and $b=3$. For $d=2$, the functions $\psi_n$ are
\[\psi_1(x,y)=\cos(\pi x)\sin(\pi y),\hspace{0.2cm} \psi_2(x)=\cos(2\pi x)\sin(\pi y),\hspace{0.2cm} \psi_3(x)=\cos(\pi x)\sin(2 \pi y),\]
the target state is $y_d(x,y)=\exp(2x+2y)\sin(2\pi x)\sin(2\pi y)$ and $a=-5$, $b=20$.
The reference solutions for $d=2$ in both unconstrained and constrained cases are shown in Fig \ref{fig:ref}.
\begin{figure}
\centering
\includegraphics[scale=0.35]{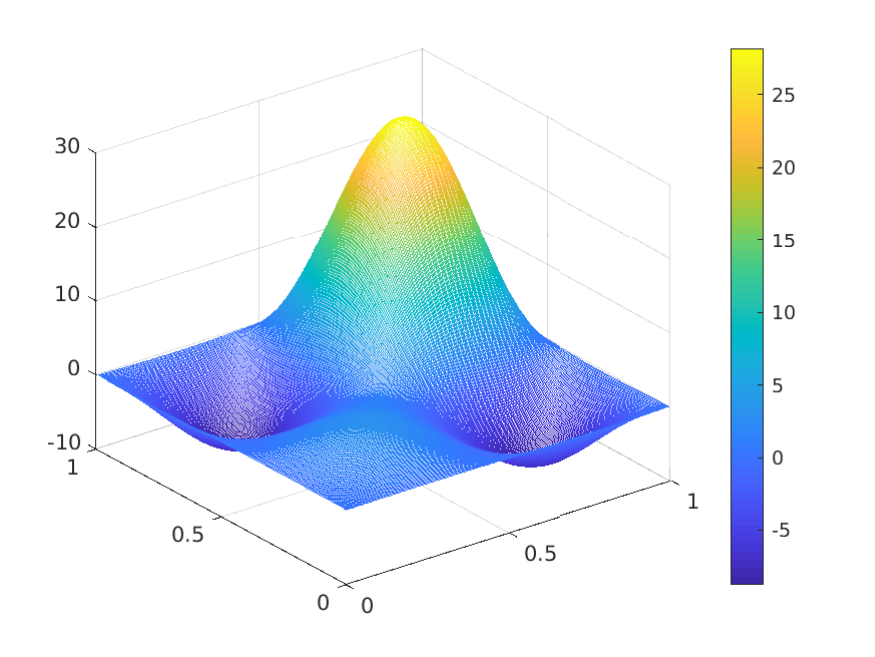}
\includegraphics[scale=0.35]{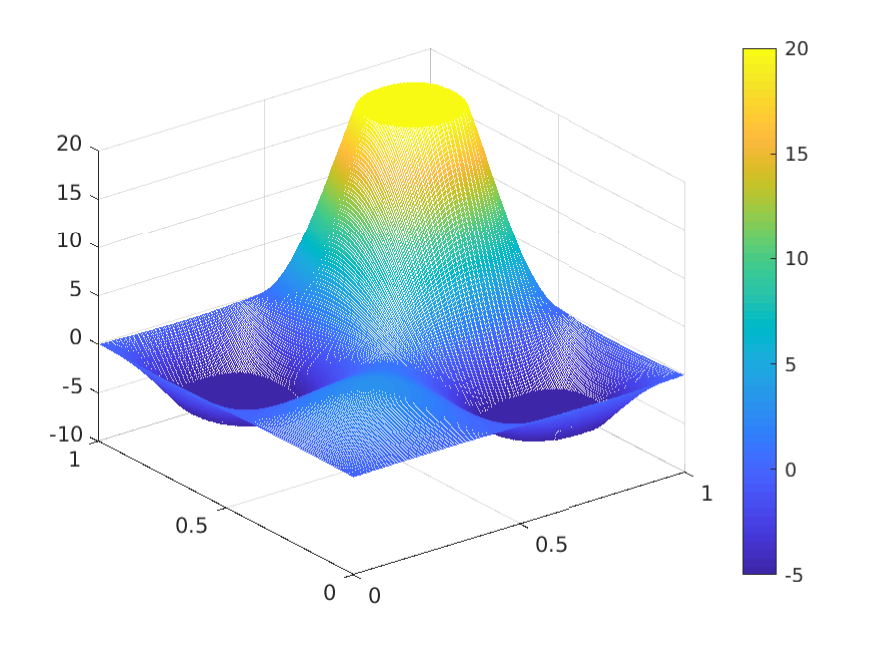}
\caption{Reference solutions  $u^\star$ for the two dimensional problem without (left) and with (right) constraints.}\label{fig:ref}
\end{figure}
\subsection{Preliminary analysis}
As preliminary step, we fit numerically the constants $C_1$, $\alpha$, $C_2$, $\beta$ and $C_3$ appearing in \eqref{eq:error_SL_complexity} and \eqref{eq:error_bound_MLMC} with the following numerical procedures. The result of this preliminary analysis are reported in Fig. \ref{fig:preliminary_analysis} for $d=2$, for the unconstrained case (first row) and unconstrained case (bottom row).
\begin{itemize}
\item To fit $C_1$ and $\alpha$, we use the same stochastic collocation discretization used for the reference solution, solve the optimization problem, and compute the $L^2(\D)$ error on the control for a sequence of spatial discretizations. We use Lagrangian continuous piecewise linear finite elements. In the unconstrained case we found $\alpha\approx 2.09$ for both the one- and two-dimensional problems, while in the constrained cased we obtain $\alpha=1.91$ in 1D, and $\alpha=1.74$ in 2D.
The value of $C_1$ depends on the setting, but in all cases a linear decay of the error with respect to the mesh size is observed in a log-log plot.
\item To fit $C_2$, we fix a fine mesh, and solve the optimization problem with a sequence of Monte Carlo quadratures with $N_{k}=2^k$ points, with $k\in \left\{2,\dots,10\right\}$. The resulting $L^2(\D)$ error on the control is then averaged over twenty repetitions. We verify that $\gamma_k=N_k^{-\frac{1}{2}}$ and again the constant $C_2$ is setting dependent. Interestingly, $C_2$ is smaller in the constrained case (intuitively, the presence of constraint reduces the variability of the adjoint variable), which in turn eases the solution of the nonlinear optimization problem by allowing a smaller number of Monte Carlo samples for a given tolerance compared to the unconstrained case.
\item To fit $\beta$ and $C_3$, we consider a sequence of meshes with $h_{\ell}\propto 2^{-\ell}$ and of quadrature formulae with $N_{\ell}\propto 4^{\ell}$. Then for each $\ell$, we minimize $J_{\ell,\ell}$, $J_{\ell-1,\ell}$ and compute numerically the quantity 
\begin{equation}\label{eq:numerical_result}
D_{\ell}:=\left\|\E\LQ p_{\ell}(\us_{\ell,\infty})-p_{\ell-1}(\us_{\ell-1,\infty})\RQ -\QQ^{\ell}\LQ p_{\ell}(\usMLMC{\ell}_{\ell,\ell})-p_{\ell-1}(\usMLMC{\ell}_{\ell-1,\ell})\RQ\right\|_{L^2(\D)},
\end{equation}
which appears in \eqref{eq:first_eq_thm} in the proof of Lemma \ref{lemma:error_bound} before we did further splitting and majorisations.
The result is averaged over ten simulations.
We obtain a rate of decay close to $3$, in both the unconstrained and constrained case, which is consistent with a choice of $\beta=2$. Numerically, we further verified the mixed decay of the controls (quantity $\Delta u$ in Lemma \ref{lemma:surplus}) in the unconstrained case, see Fig. \ref{fig:time_solution} (left). With box constraints, the claim of Lemma \ref{lemma:surplus} seems numerically not to hold. Nevertheless, this does not prevent $D_{\ell}$ to exhibit the correct mixed decay which justifies the application of the MLMC quadrature formula even with control constraints.
\end{itemize}

\begin{figure}
\includegraphics[scale=0.33]{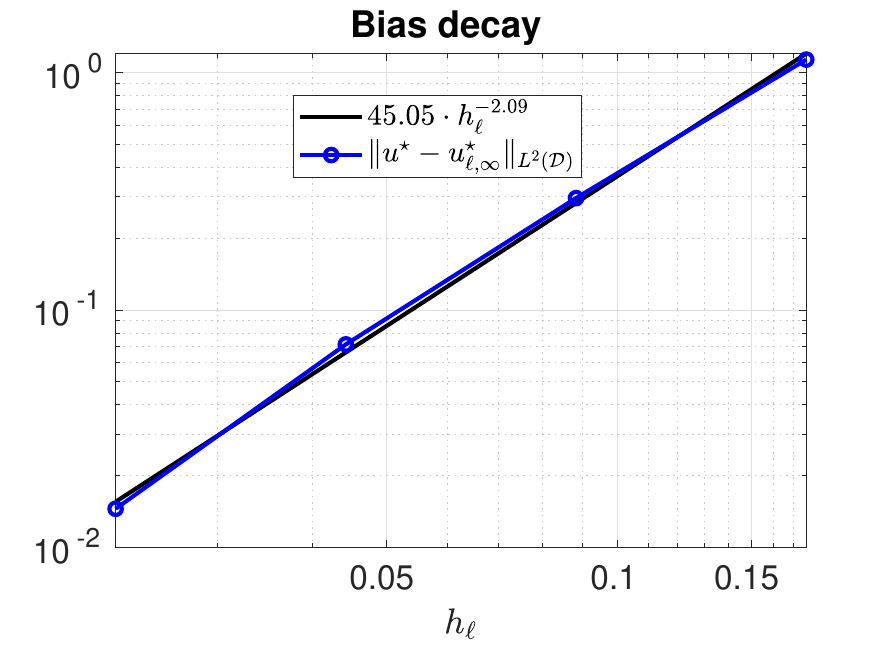}\includegraphics[scale=0.33]{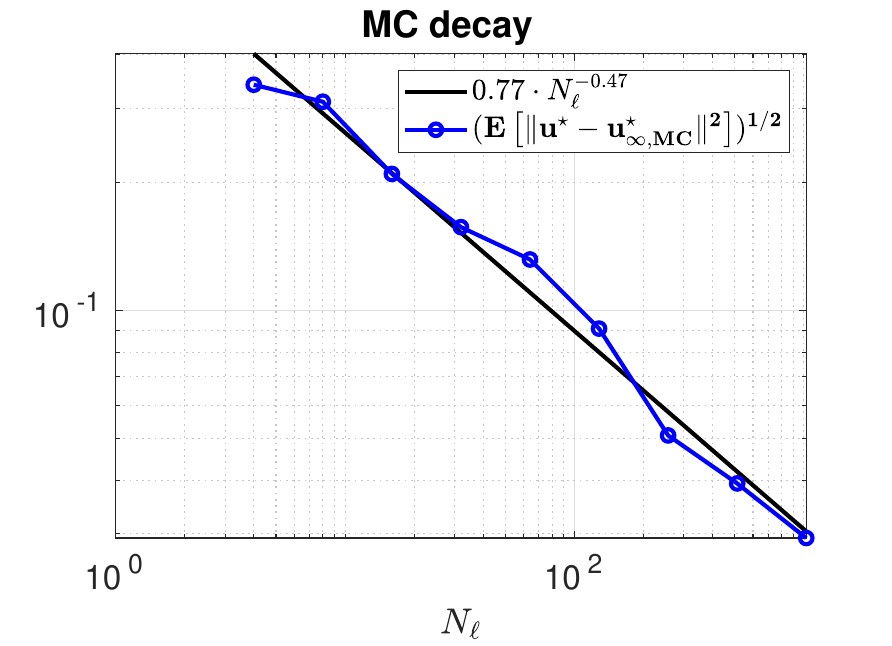}\includegraphics[scale=0.33]{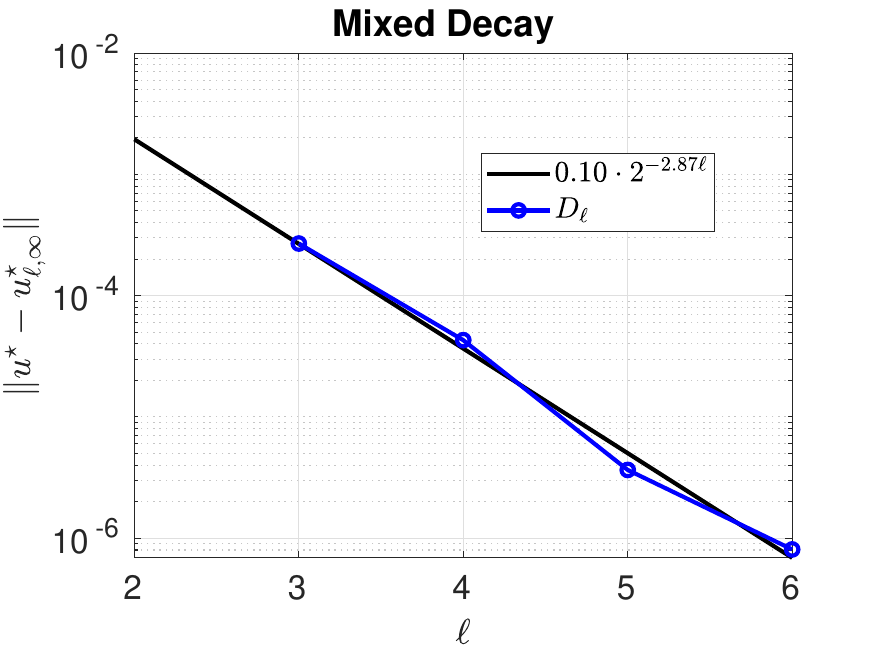}\\
\includegraphics[scale=0.33]{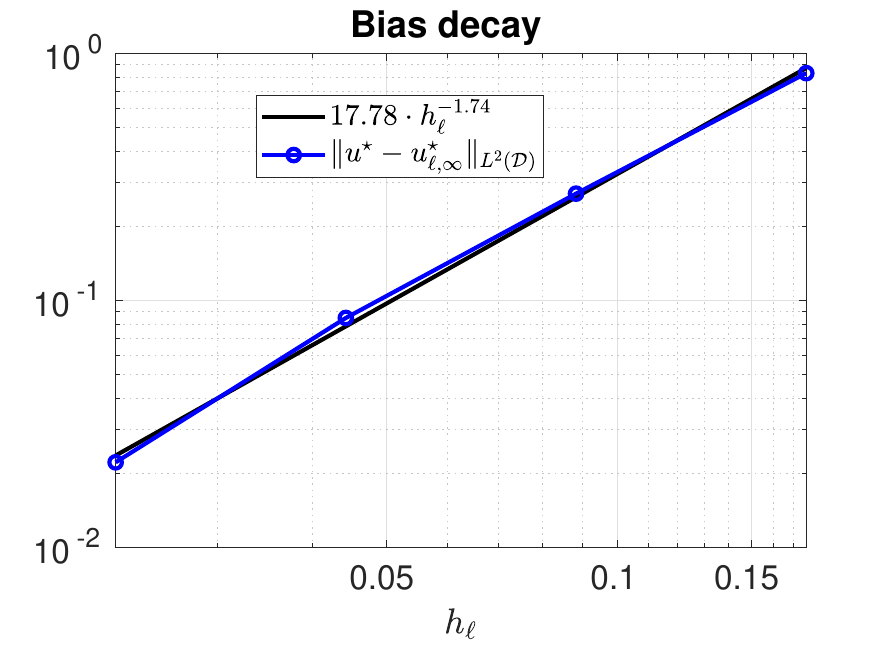}\includegraphics[scale=0.33]{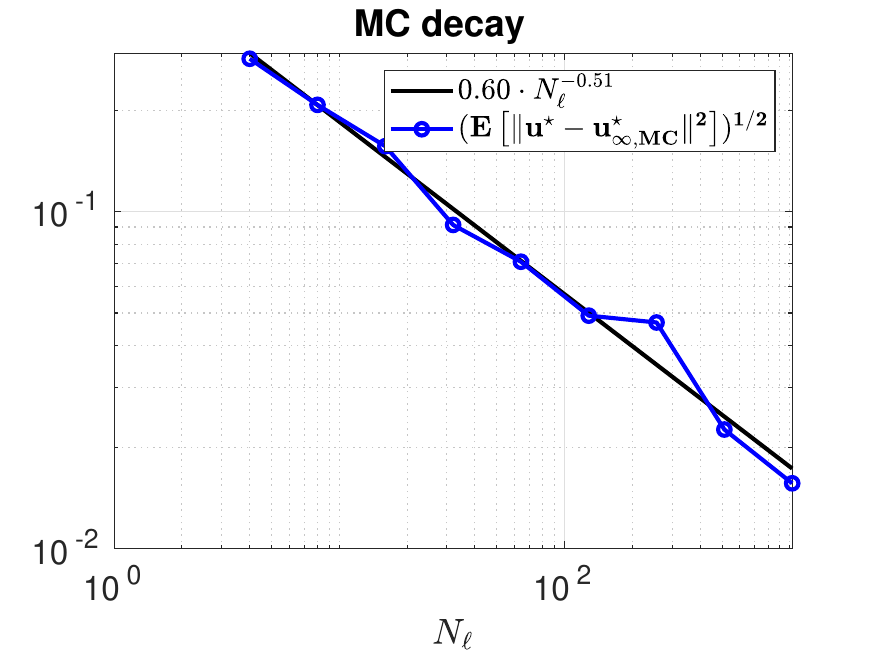}\includegraphics[scale=0.33]{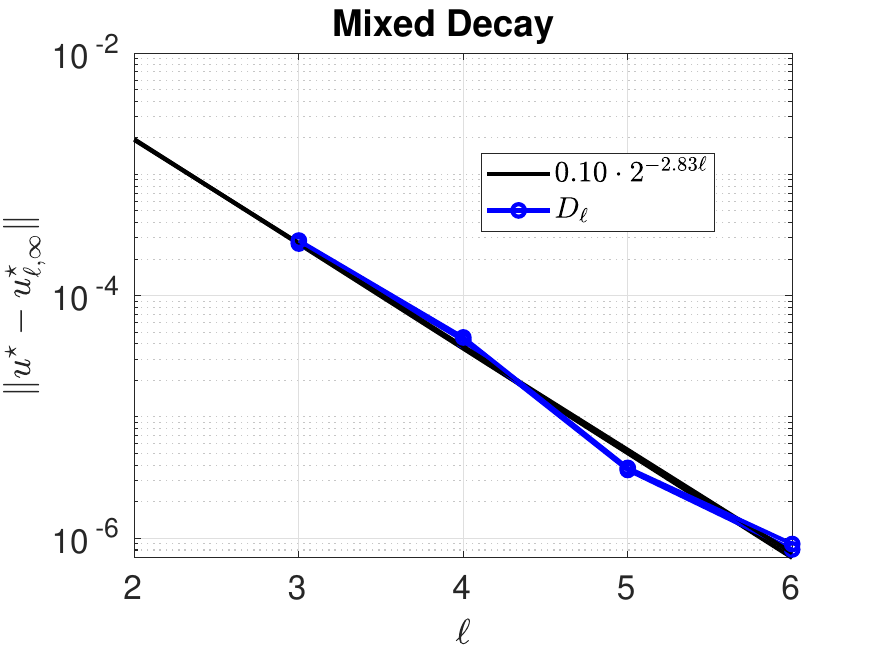}\\
\caption{Results of the preliminary analysis to fit the constants $C_1$, $\alpha$, $C_2$, $\beta$ and $C_3$ for the unconstrained case (top row) and constrained case (bottom row) for $d=2$.}\label{fig:preliminary_analysis}
\end{figure}

As final preliminary step, we verify that the cost of minimizing $J_{\ell,k}$ is linear with respect to $\text{dim}(V_{\ell})N_k$. In the unconstrained case, $J_{\ell,k}$ is minimized by solving directly the full-space optimality system \eqref{eq:full_space_system_discretized} using the block diagonal preconditioner proposed in \cite{Kouri2018} and analyzed in \cite{vanzan}. In particular such preconditioner allows to precondition in parallel the $2N_k$ PDEs, and leads to a constant number of Krylov iterations with respect to the mesh size and number of samples. Fig \ref{fig:time_solution} shows the computational time to assemble all the finite element matrices and to solve the optimality system as the size of the problem increases. 
In the constrained case, we solve the minimization problem with an outer nonlinear semismooth Newton iteration, and as inner solver we use again a Krylov method preconditioned by the same block diagonal preconditioner as for the unconstrained problem. In \cite{ciaramella2024multigrid}, the authors showed that this strategy leads to a robust number of outer nonlinear iterations with respect to the parameters of interest, so that the overall computational cost can be considered linear in the size of the problem as in the unconstrained case, but with a larger constant. 
\begin{figure}
\centering
\includegraphics[scale=0.33]{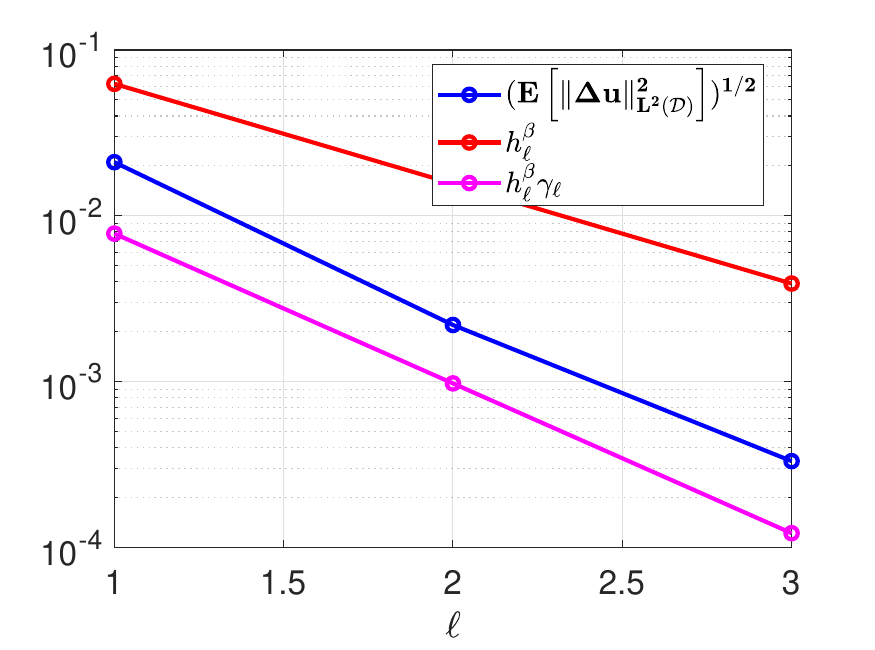}
\includegraphics[scale=0.33]{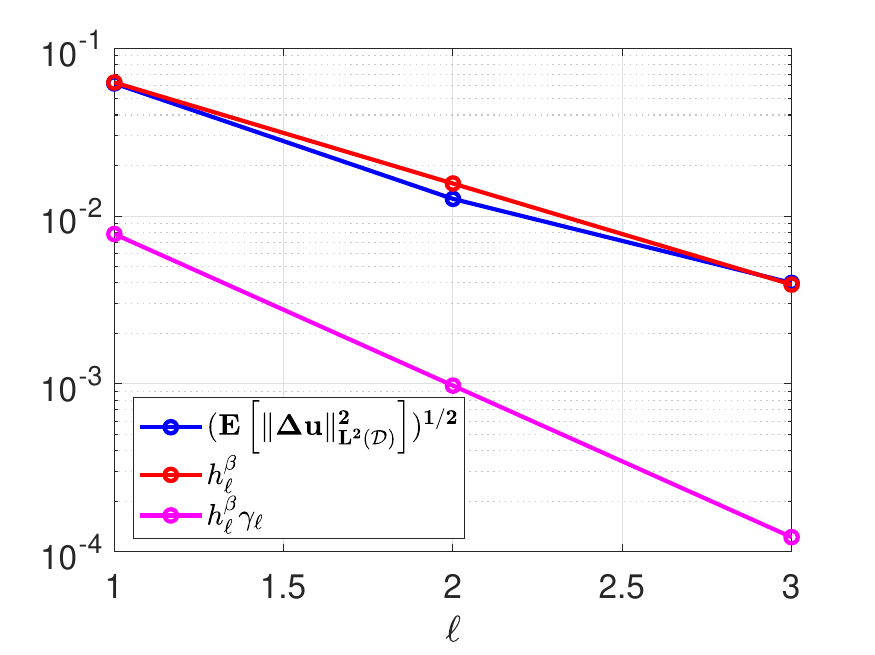}
\includegraphics[scale=0.33]{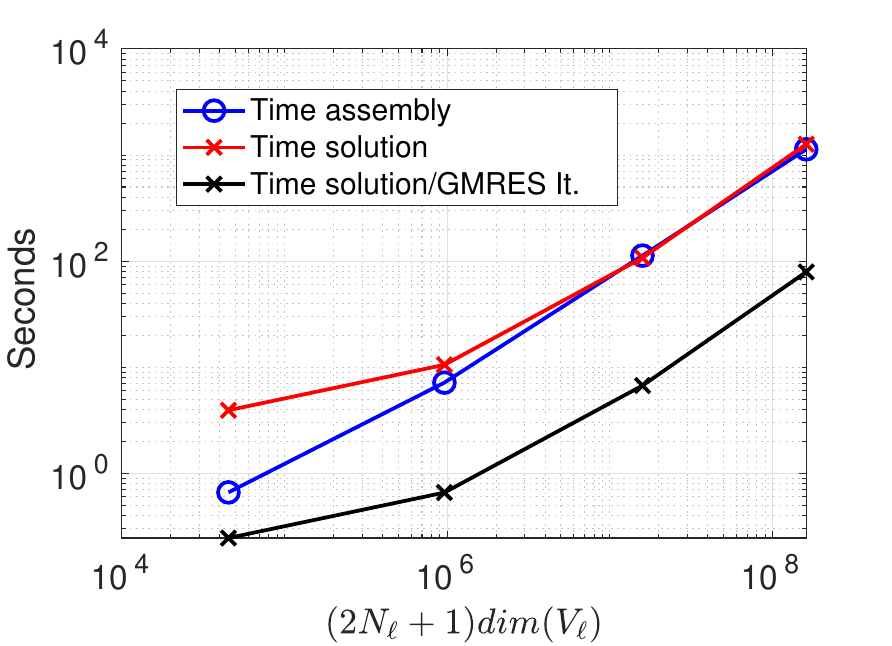}
\caption{Verification of \eqref{eq:lemma} without (left) and with (center) box constraints. The right panel shows the growth of the computational time to assemble the finite element matrices and to solve the optimality system as the problem size increases.}\label{fig:time_solution}
\end{figure}

\subsection{Complexity results}

We finally compare the MC method and the MLMC method described by \eqref{eq:MLMC} both in term of complexity to reach a given tolerance $\varepsilon$ and of related computational times.
For the unconstrained case, Fig. \ref{fig:unconstrained} shows for $d=1$ (top row) and $d=2$ (bottom row): the linear decay of the error with respect to a prescribed tolerance (left column), the related complexity measured by $\text{dim}{V_{L}}(2N_{K}+1)$ (with $L$ and $N_K$ chosen as in \eqref{eq:L_N_k}) for the MC method and by $\sum_{\ell=0}^L \text{dim}{V_{\ell}}(2N_{\ell}+1)$ for MLMC method (center column), and the overall computational time measured in seconds (right column). Fig. \ref{fig:constrained} replicates Fig. \ref{fig:unconstrained} for the constrained case.

\begin{figure}
\centering
\includegraphics[scale=0.33]{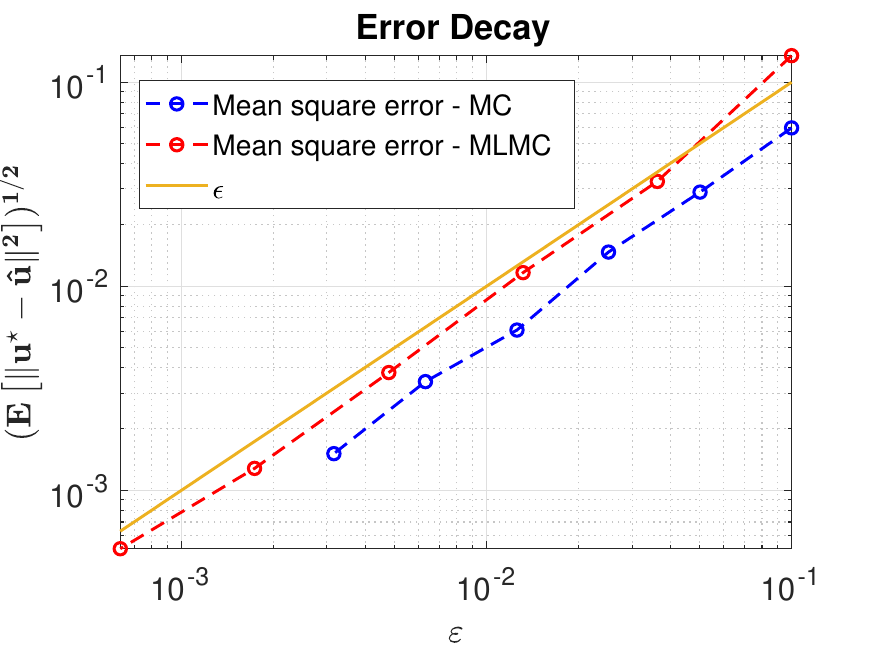}.pdf\includegraphics[scale=0.33]{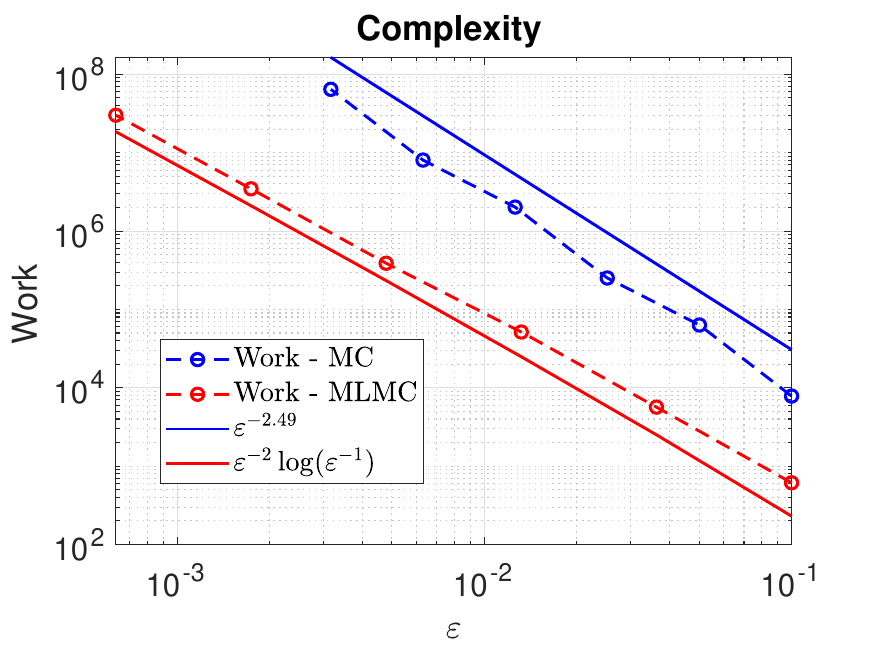}\includegraphics[scale=0.33]{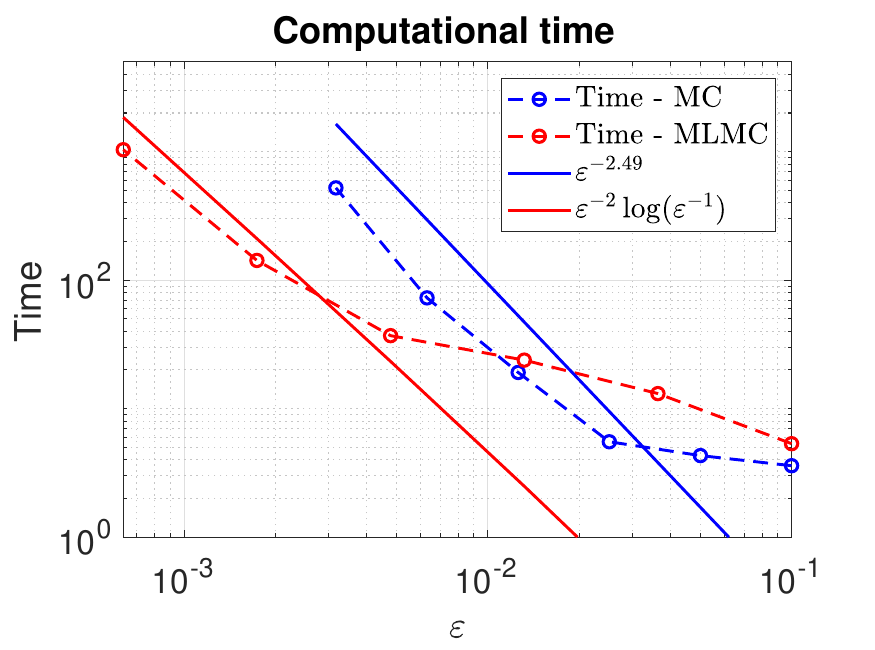}\\
\includegraphics[scale=0.33]{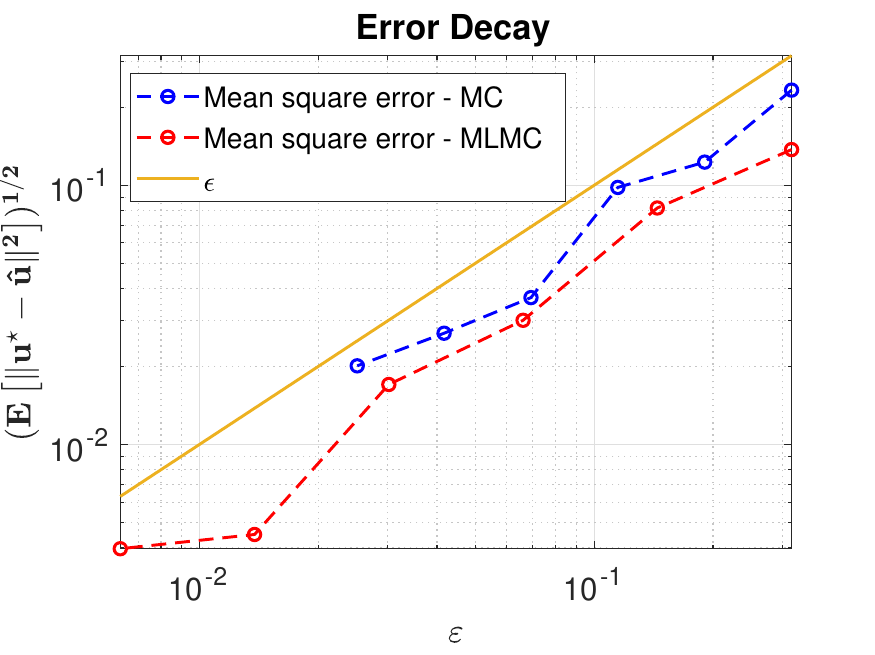}\includegraphics[scale=0.33]{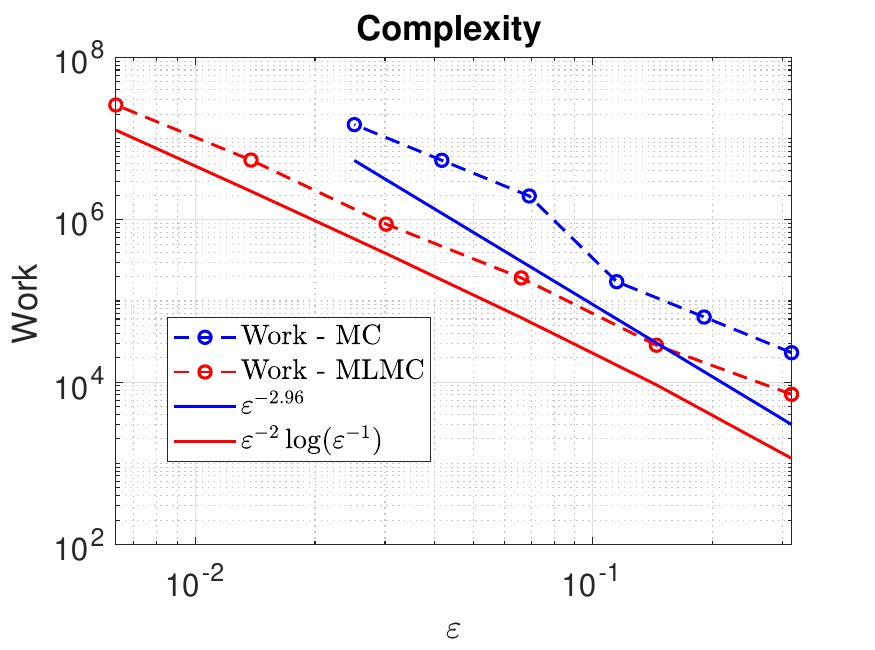}\includegraphics[scale=0.33]{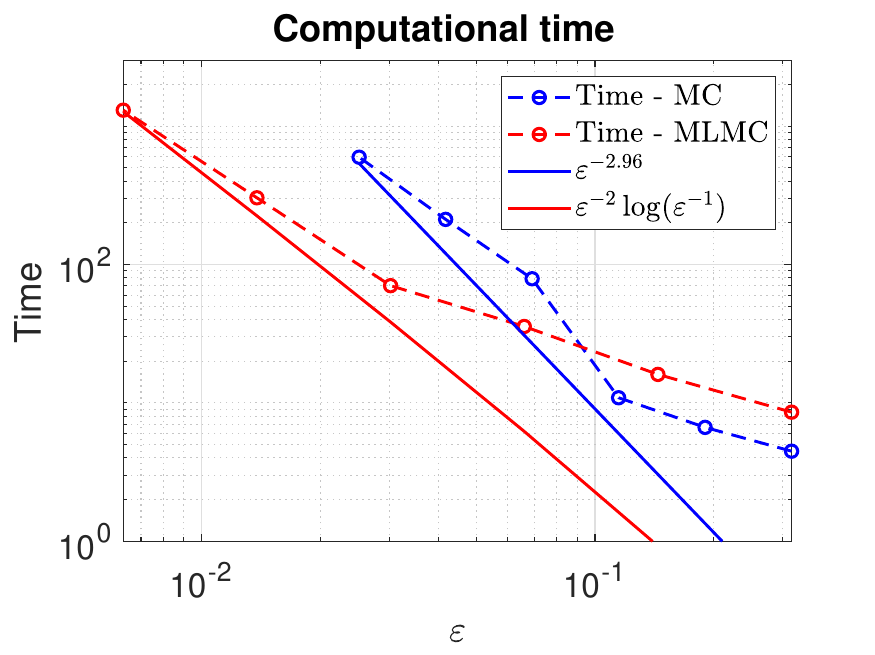}
\caption{For the unconstrained problem, we report the error decay (left panel), the computational complexity (center panel) and the computational time (right panel). The top row refers to $d=1$, the bottom row to $d=2$.}\label{fig:unconstrained}
\end{figure}
\begin{figure}
\centering
\includegraphics[scale=0.33]{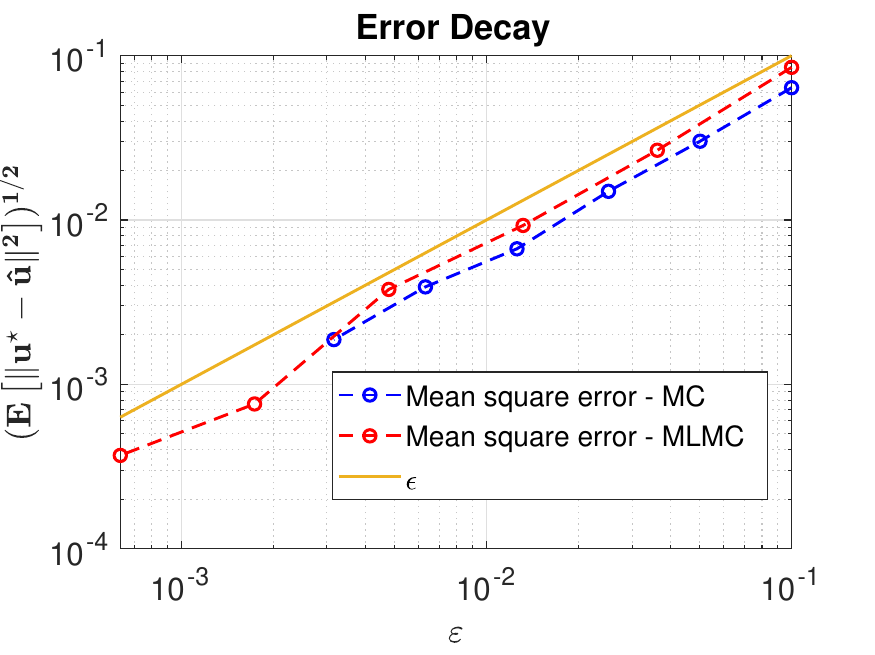}\includegraphics[scale=0.33]{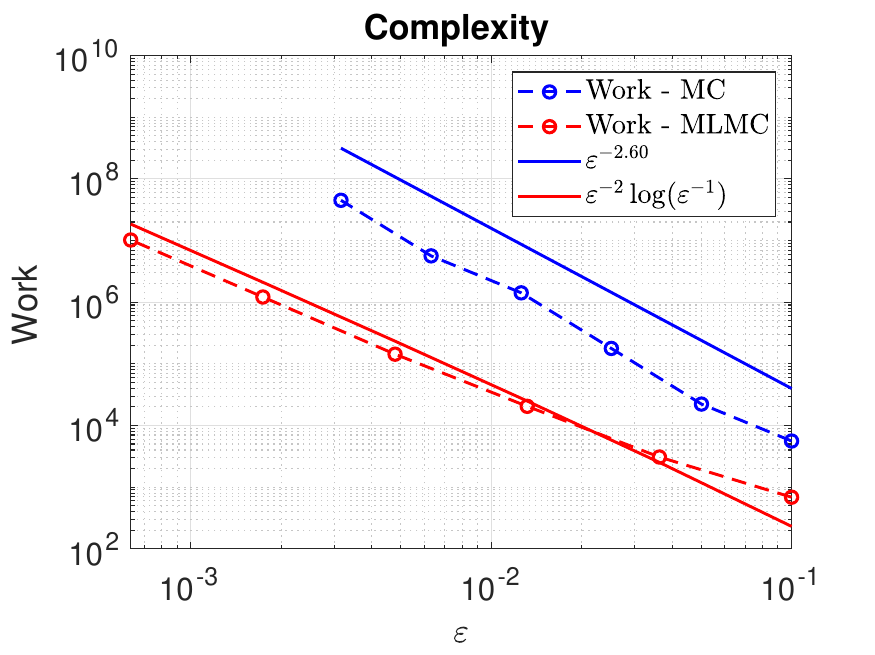}\includegraphics[scale=0.33]{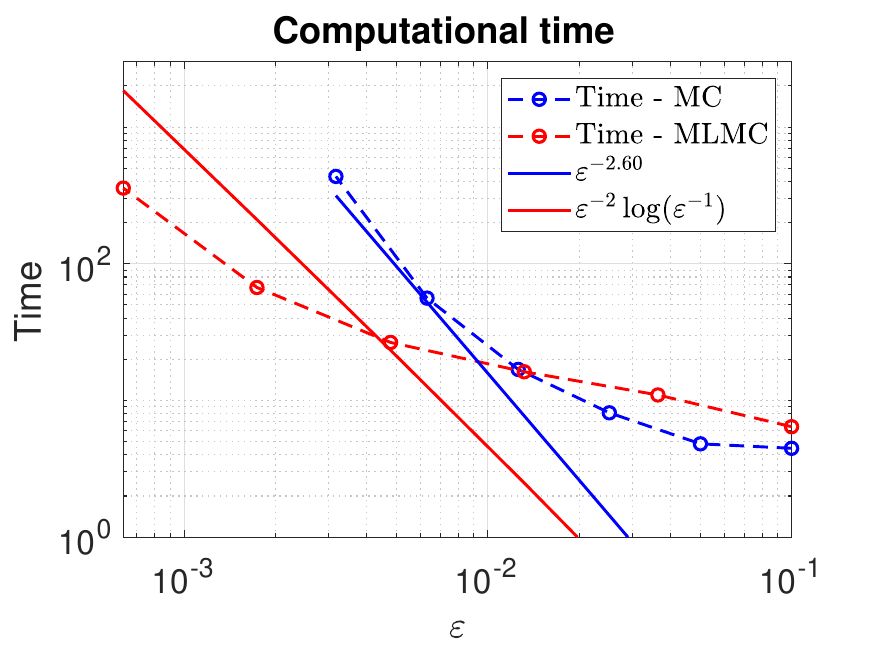}\\
\includegraphics[scale=0.33]{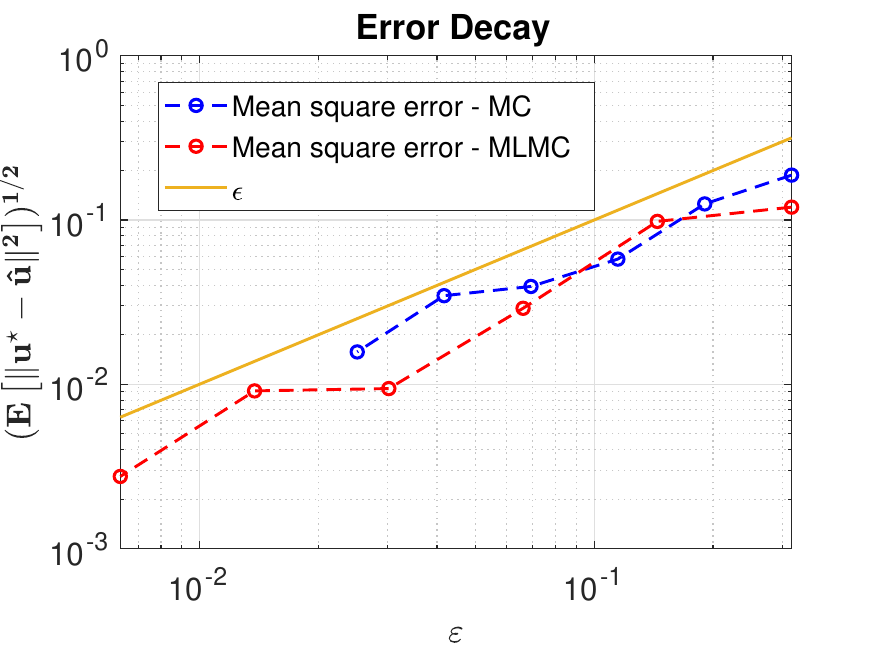}\includegraphics[scale=0.33]{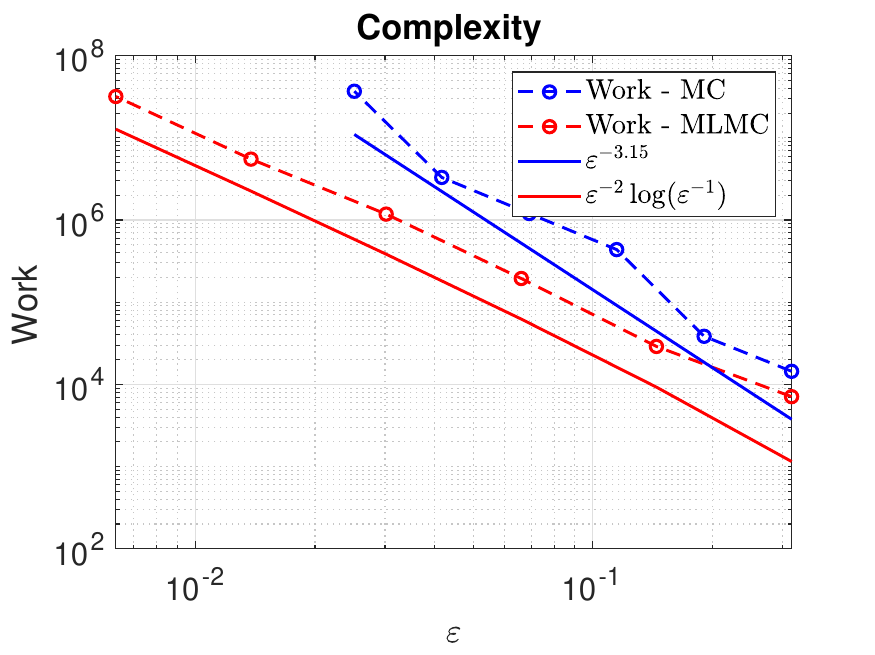}\includegraphics[scale=0.33]{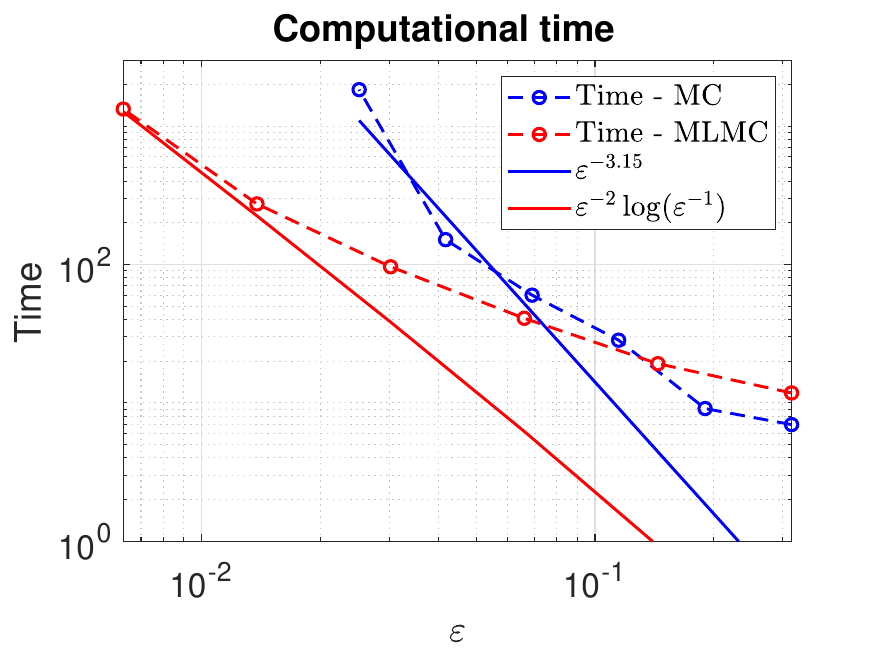}
\caption{For the constrained problem, we report the error decay (left panel), the computational complexity (center panel) and the computational time (right panel). The top row refers to $d=1$, the bottom row to $d=2$.}
\label{fig:constrained}
\end{figure}

Figures \ref{fig:unconstrained} and \ref{fig:constrained} shows the reduced computational cost of the proposed method compared to a standard MC sample average approximation, and confirm the complexity analysis presented in Section \ref{sec:Complexity}, in particular the simplified result \eqref{eq:complexity_MLMC}.
We further emphasize that in all simulations we could not reach a smaller tolerance with the MC method, since the single, very large, optimization problem quickly saturates the memory of the workstation as $\varepsilon$ decreases. In constrast, the MLMC method allows us to reach smaller tolerances by solving a sequence of smaller optimization problems that more easily satisfy the limits of our workstation.

\section{Conclusions}
In this manuscript, we presented a methodology to use multilevel quadrature formulae in the context of OCPs under uncertainty. It consists in solving a sequence of minimization problems discretized with different levels of accuracy, and in combining a-posteriori the approximated quadratures of the discretized adjoint variables.
We provided a convergence and complexity analysis for an unconstrained linear quadratic optimization problem, and numerical experiments confirm the better complexity compared to common sample-average approaches, even with box constraints.

Our multilevel technique can in principle readily accomodate nonlinear PDE constraints and general objective functionals, but currently we do not have a comprehensive convergence result. Intuitively, we expect that the reduced functional should be locally convex around the optimal solution, so that for sufficiently fine discretizations, the approximated controls remain close to the exact one.

General risk measures can also be considered, as long as they lead to an optimality condition requiring the control to be equal to an expectation of (possibly, a continuous function of) the adjoint variable. The smoothed Conditional Value at Risk (CVAR) satisfies this condition, while the discontinuity of the unsmoothed CVAR might be potentially handled with techniques from \cite{giles2019multilevel}.
These extensions will be the subject of future works.

\section*{Acknowledgments}
The second author is member of the INDAM-GNCS group.

\end{document}